\documentclass[12pt,A4]{article}
\usepackage[french,english]{babel}
\usepackage{amsmath,amsfonts,amssymb,amsthm,mathrsfs,bbm}
\usepackage[font=sf, labelfont={sf,bf}, margin=1cm]{caption}
\usepackage{graphicx,graphics}
\usepackage{epsfig}
\usepackage{latexsym}
\usepackage[applemac]{inputenc}
\usepackage{ae,aecompl}
\usepackage{textcomp}
\usepackage{pstricks}
\usepackage{xfrac}
\usepackage{enumerate}
\usepackage{xcolor}
\usepackage[colorlinks=true,urlcolor=blue,citecolor=blue,citebordercolor=blue,linkcolor=blue]{hyperref}
\usepackage[top=2.5cm,left=1.5cm,right=1.5cm,bottom=2.5cm]{geometry}
\usepackage{comment}
\usepackage{eulervm,palatino}

\DeclareGraphicsExtensions{.jpg,.pdf}

\newcommand{\op}[1]{\operatorname{#1 }}

\newcommand{\E}[1]{\mathbb{E} \left [ #1 \right ]}
\newcommand{\R}{\mathbb R}

\newcommand{\Z}{\mathbb Z}

\renewcommand{\P}[2]{\mathbb{P}_{#2}\left( #1 \right)}

\def\build#1_#2^#3{\mathrel{
\mathop{\kern 0pt#1}\limits_{#2}^{#3}}}

\newtheorem{theorem}{Theorem}[section]

\newtheorem{proposition}[theorem]{Proposition}
\newtheorem{lemma}[theorem]{Lemma}
\newtheorem{corollary}[theorem]{Corollary}
\newtheorem{definition}[theorem]{Definition}

\renewcommand{\P}[1]{\mathbb{P}\left[#1\right]}
\newtheorem{rek}[theorem]{Remark}

\theoremstyle{definition}

\newcommand\Es[1]{\mathbb{E}\left[#1\right]}
\renewcommand\Pr[1]{\mathbb{P}\left(#1\right)}
\newcommand\PrB[1]{\mathbb{P}_{b}\left(#1\right)}

\newcommand \GWE[1]{\mathsf{GW}_{\mu^{ \circ},\mu^{\bullet}} \left[#1\right]}
\newcommand \GWnu[1]{\mathsf{GW}_{ \nu} \left(#1\right)}
\newcommand \GWrho[1]{\mathsf{GW}_{ \rho} \left(#1\right)}
\newcommand \Esrho[1]{\mathsf{GW}_{ \rho} \left[#1\right]}
\newcommand \Esnu[1]{\mathsf{GW}_{ \nu} \left[#1\right]}

\def \D {\mathbb D}

\def \R {\mathbb R}

\def \E {\mathbb E}
\def \P {\mathbb{P}}

\newcommand \fl[1] {\left\lfloor #1 \right\rfloor}

\usepackage{enumerate}

\title{ \textsc{Percolation on random triangulations \\ and  \\ stable looptrees}}
\author{Nicolas Curien\thanks{CNRS and LPMA UPMC Paris 6, E-mail: nicolas.curien@gmail.com}  \quad \& \quad  Igor Kortchemski\thanks{DMA, École Normale Sup\'erieure, E-mail: igor.kortchemski@normalesup.org} }

\date{}

\DeclareSymbolFont{extraup}{U}{zavm}{m}{n}
\DeclareMathSymbol{\varheart}{\mathalpha}{extraup}{86}
\DeclareMathSymbol{\vardiamond}{\mathalpha}{extraup}{87}

\makeatletter
\renewcommand*{\@fnsymbol}[1]{\ensuremath{\ifcase#1\or  \vardiamond \or \clubsuit\or \spadesuit\or
   \mathsection\or \mathparagraph\or \|\or **\or \dagger\dagger
   \or \ddagger\ddagger \else\@ctrerr\fi}}
\makeatother

\begin{document}
\maketitle
\let\thefootnote\relax\footnotetext{ \\\emph{MSC2010 subject classiﬁcations}. Primary  05C80, 60J80 ; secondary 60K35.
 \emph{This work is partially supported by the French “Agence Nationale de la Recherche”  ANR-08-BLAN-0190.}}
  
\abstract{We study site percolation on Angel \& Schramm's  Uniform Infinite Planar Triangulation. We compute several critical and near-critical exponents, and describe the scaling limit of the boundary of large percolation clusters in all regimes  (subcritical, critical and supercritical). We prove in particular that the scaling limit of the boundary of large critical percolation clusters is the random stable looptree of index $3/2$, which was introduced in \cite{CK13}.
We also give a conjecture linking looptrees of any index $\alpha \in (1,2)$ with scaling limits of cluster boundaries in random triangulations decorated with $O(N)$ models.}

\vfill

\begin{figure}[!h]
 \begin{center}
 \includegraphics[width=0.45 \linewidth]{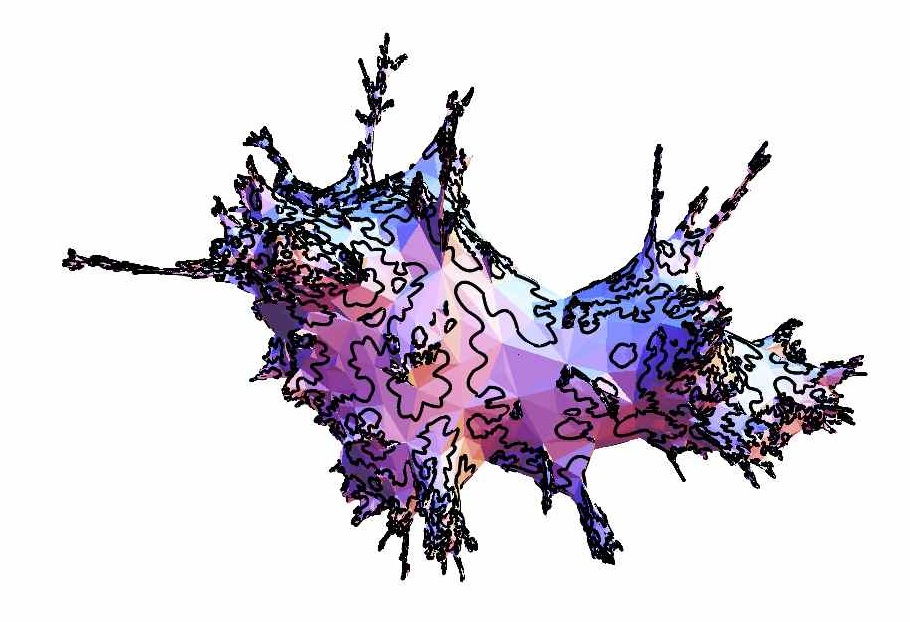}
 \caption{A site percolated triangulation and the interfaces separating the clusters.}
 \end{center}
 \end{figure}
 
 \vfill

 \clearpage

\section{Introduction}
We investigate site percolation on large random triangulations and in particular on the Uniform Infinite Planar Triangulation (in short, UIPT) which was introduced by Angel \& Schramm \cite{AS03}. In particular, we compute the critical and near-critical exponents related to the perimeter of percolation interfaces and we identify the scaling limit of the boundary of large clusters in all regimes (critical, subcritical and supercritical). In the critical case, this limit is shown to be $ \mathscr{L}_{3/2}$, the stable looptree of parameter ${3}/{2}$ introduced in \cite{CK13}. Our method is based on a surgery technique inspired from Borot, Bouttier \& Guitter and on a tree  decomposition of 
triangulations with non-simple boundary. We finally state precise conjectures linking the whole family of looptrees $( \mathscr{L}_{\alpha})_{1<\alpha<2}$ to scaling limits of cluster boundaries of random planar triangulations decorated with $O(N)$ models.

\paragraph{The UIPT.} The probabilistic theory of random planar maps and its physics counterpart, the Liouville  2D quantum gravity, is a very active field of research. See in particular the work of Le Gall and Miermont on scaling limits of large random planar maps and the Brownian map \cite{LG11,Mie11}. The goal is to understand universal large-scale properties of random planar graphs or maps.
One possible way to get information about the geometry of these random lattices is to understand the behavior of (critical) statistical mechanics models on them. In this paper, we focus on one of the simplest of such models: site percolation on random triangulations.  

Recall that a triangulation is a proper embedding of a finite connected graph in the two-dimensional sphere, considered up to orientation-preserving homeomorphisms, and such that all the faces have degree $3$. We only consider \emph{rooted} triangulations, meaning that an oriented edge is distinguished and called the root edge. Note that we allow loops and multiple edges.
We write $ \mathbb{T}_{n}$ for the set of all rooted triangulations with $n$ vertices, and let $T_{n}$ be a random triangulation chosen uniformly at random among $ \mathbb{T}_{n}$. Angel \& Schramm \cite{AS03} have introduced an infinite random planar triangulation $T_{\infty}$, called the Uniform Infinite Planar Triangulation (UIPT), which is obtained as the local limit of $T_{n}$ as $ n \rightarrow \infty$. More precisely, $T_{\infty}$ is characterized by the fact that for every $r \geq 0$ we have the following convergence in distribution
 \begin{eqnarray} \label{eq:charT}B_{r}(T_{n}) & \xrightarrow[n\to\infty]{(d)} & B_{r}(T_{\infty}),  \end{eqnarray}
where $B_r(m)$ is the map formed by the edges and vertices of $m$ that  are at
graph distance smaller than or equal to $r$ from the origin of the root edge.  This infinite random triangulation and its quadrangulation analog (the UIPQ, see \cite{CD06,Kri05}) have attracted a lot of attention, see \cite{BCsubdiffusive,CurKPZ,GGN12} and the references therein. 

\paragraph{Percolation.} Given  the UIPT, we consider a site percolation by coloring its vertices independently white with probability $a\in (0,1)$ and black with probability $1-a$. This model has already been studied by Angel \cite{Ang03}, who proved that the critical threshold  is almost surely equal to $$a_{c} =   {1}/{2}.$$ His approach was based on a clever Markovian exploration of the UIPT called the peeling process. See also \cite{ACpercopeel,CurKPZ,MN13} for further studies of  percolation on random maps using the peeling process.  \medskip

In this work, we are interested in the geometry of \emph{the boundary of percolation clusters} and use a different approach. We condition the root edge of the UIPT on being of the form $\circ \to \bullet$, which will allow us to define the percolation interface going through the root edge. The white cluster of the origin is by definition the set of all the white vertices and edges between them that can be reached from the origin of the root edge using white vertices only.  We denote by $ \mathcal{H}^{\circ}_{a}$ its hull, which is obtained by filling-in the holes of the white cluster except the one containing the target of the root edge called the exterior component (see Fig.~\ref{percoex-site} and Section \ref{sec:necklace} below for a precise definition). Finally,  we denote by  $ \partial  \mathcal{H}_{a}^\circ$ the boundary of the hull, which is the graph formed by the edges and vertices of  $\mathcal{H}^{\circ}_{a}$ adjacent to the exterior (see Fig.~\ref{percoex-site}), and let $ \#\partial  \mathcal{H}_{a}^\circ$ be its perimeter, or length, that is the number of half edges of $\partial  \mathcal{H}_{a}^\circ$ belonging to the exterior. Note that  $  \partial \mathcal{H}^{\circ}_{a}$ is formed of discrete cycles attached by some pinch-points.  It follows from the work of Angel \cite{Ang03} that, for every value of $a \in (0,1)$, the boundary $\partial \mathcal{H}_{a}^\circ$ is always \emph{finite} (if an infinite interface separating a black cluster from a white cluster existed, this would imply the existence of both  infinite black and white clusters, which is intuitively not posible).

One of our contributions is to find the precise asymptotic behavior for the probability of having a large perimeter in the critical case.

\begin{figure}[!h]
 \begin{center}
 \includegraphics[width=0.8 \linewidth]{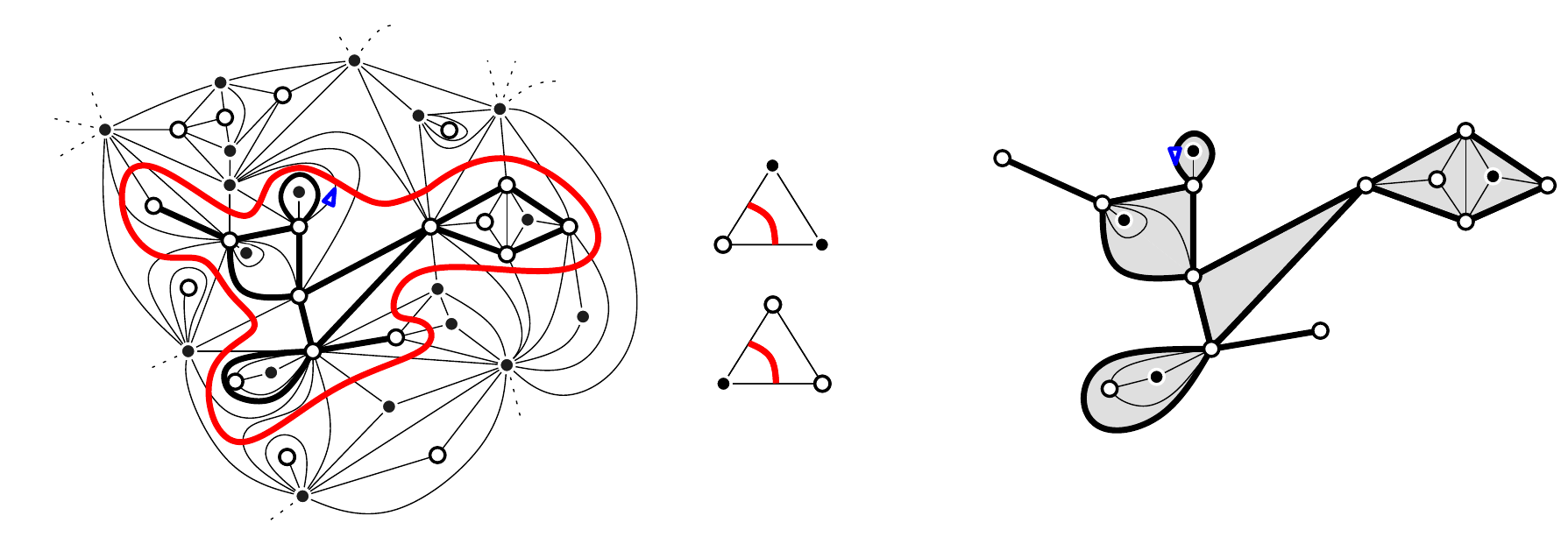}
 \caption{\label{fig:clusterloop}On the left, a part of a site-percolated triangulation with the interface going through the root edge and the hull of the cluster of the origin. The interface is drawn using the rules displayed in the middle triangles. On the right, the boundary of the hull is in bold black line segments and has perimeter $16$.\label{percoex-site}} 
 \end{center}
 \end{figure}

 \begin{theorem}[Critical exponent for the perimeter] \label{thm:expo}For $a= a_{c}=1/2$ we have 
 $$ \Pr{ \# \partial \mathcal{H}^{ \circ}_{\text{\sfrac{1}{2}}} = n} \quad  \underset{n \to \infty}{\sim} \quad \frac{ {3}}{ {2} \cdot |\Gamma(-{2}/{3})|^3} \cdot n^{-4/3},$$
 where $ \Gamma$ is Euler's Gamma function.
 \end{theorem}

It is interesting to mention that the exponent $4/3$ for the perimeter of the boundary of critical clusters also appears when dealing with the half-plane model of the UIPT: using the peeling process, it is shown in \cite{ACpercopeel} that $ P( \# \partial \mathcal{H}_{a_{c}}^\circ > n) \asymp n^{-1/3}$, where  $a_{n} \asymp b_{n}$ means that the sequence $a_{n}/b_{n}$ is bounded from below and above by certain constants.

The main idea used to establish Theorem \ref{thm:expo} is a tree representation of the $2$-connected components of $  \partial \mathcal{H}_{a}^{\circ}$, which we prove to be closely related to the law of a certain two-type Galton--Watson tree. We reduce the study of this two-type random tree to the study of a standard one-type Galton--Watson tree by using a recent bijection due to Janson \& Stef\'ansson \cite{JS12},  which enables us to use the vast literature on random trees and branching processes to make exact computations.

This method also allows us to fully understand the probabilistic structure of the hull of the white cluster and to identify the scaling limits (for the Gromov--Hausdorff topology) in any regime (subcritical, critical and supercritical)  of $\partial \mathcal{H}_{a}^{\circ}$, seen as a compact metric space, when its perimeter tends to infinity. In particular, we establish that the scaling limit of $\partial \mathcal{H}_{a_{c}}^{\circ}$ conditioned to be large, appropriately rescaled, is the stable looptree of parameter $ 3/2$ introduced in \cite{CK13}, whose definition we now recall.

\paragraph{Stable looptrees.} Random stable looptrees are random compact metric spaces and can, in a certain sense, be seen as the dual of the stable trees introduced and studied in \cite{DLG02,LGLJ98}. They are constructed in \cite{CK13} using stable processes with no negative jumps, but can also be defined as scaling limits of discrete objects: With every rooted oriented tree (or plane tree) $ \tau$, we  associate a graph, called the discrete looptree of $ \tau$ and denoted by $ \mathsf{Loop}( \tau)$,  which is the graph on the set of vertices of $\tau$ such that two vertices $u$ and $v$ are joined by an edge if and only if one of the following three conditions are satisfied in $ \tau$: $u$ and $v$ are consecutive siblings of a same parent, or $u$ is the first sibling (in the lexicographical order) of $v$, or $u$ is the last sibling of $v$, see Fig.~\ref{fig:loop}. Note that in \cite{CK13},  $\mathsf{Loop}( \tau)$ is defined as a different graph, and that here $\mathsf{Loop}( \tau)$ is the graph which is denoted by $\mathsf{Loop}'( \tau)$ in \cite{CK13}. We view $ \mathsf{Loop}( \tau)$ as a compact metric space by endowing its vertices with the graph distance (every edge has unit length).

 \begin{figure}[!h]
 \begin{center}
 \includegraphics[width=0.65  \linewidth]{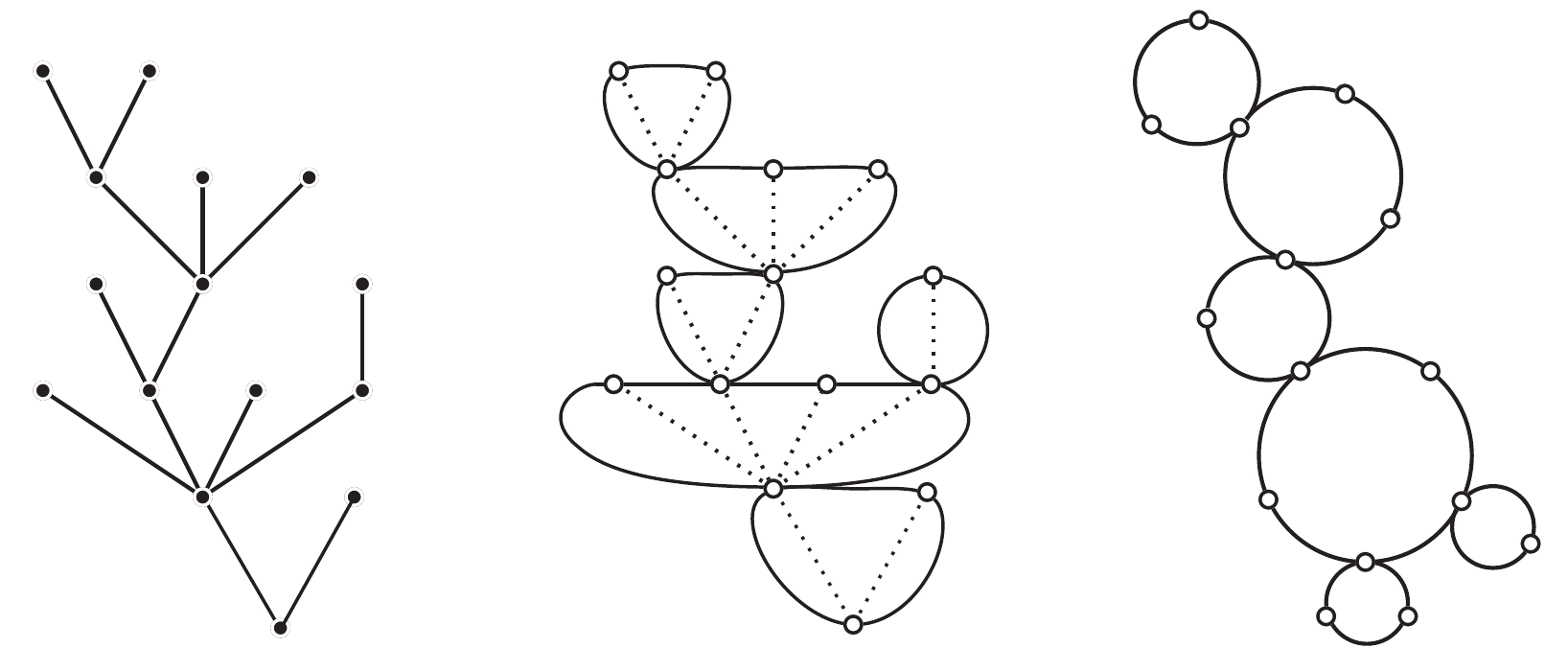}
 \caption{ \label{fig:loop}A plane tree $\tau$ (left) and its looptree $ \mathsf{Loop}( \tau)$ (middle and right).}
 \end{center}
 \end{figure} 
  Fix $ \alpha \in (1,2)$.  Now let $\tau_{n}$ be a Galton--Watson tree conditioned on having $n$ vertices, whose offspring distribution $\mu$ is critical and satisfies $\mu_{k} \sim c \cdot k^{-1-\alpha}$ as $k \rightarrow \infty$ for a certain $c>0$.   In \cite[Section 4.2]{CK13}, it is shown that there exists a random compact metric space $\mathscr{L}_{\alpha}$, called the stable looptree of index $ \alpha$, such that 
\begin{eqnarray} \label{eq:invprinc}  n^{-1/\alpha} \cdot \mathsf{Loop}( \tau_{n}) & \quad \xrightarrow[n\to\infty]{(d)} \quad &  \left(c |\Gamma(- \alpha)| \right)^{-1/ \alpha} \cdot \mathscr{L}_{\alpha}, \end{eqnarray}
where the convergence holds in distribution for the Gromov--Hausdorff topology and where $c \cdot M$ stands for the metric space obtained from $M$ by multiplying all distances by $c >0$. Recall that the Gromov--Hausdorff topology gives a sense to the convergence of (isometry classes) of compact metric spaces, see Section \ref{sec:GH} below for the definition.

 \begin{figure}[!h]
 \begin{center}
  \includegraphics[height=5cm]{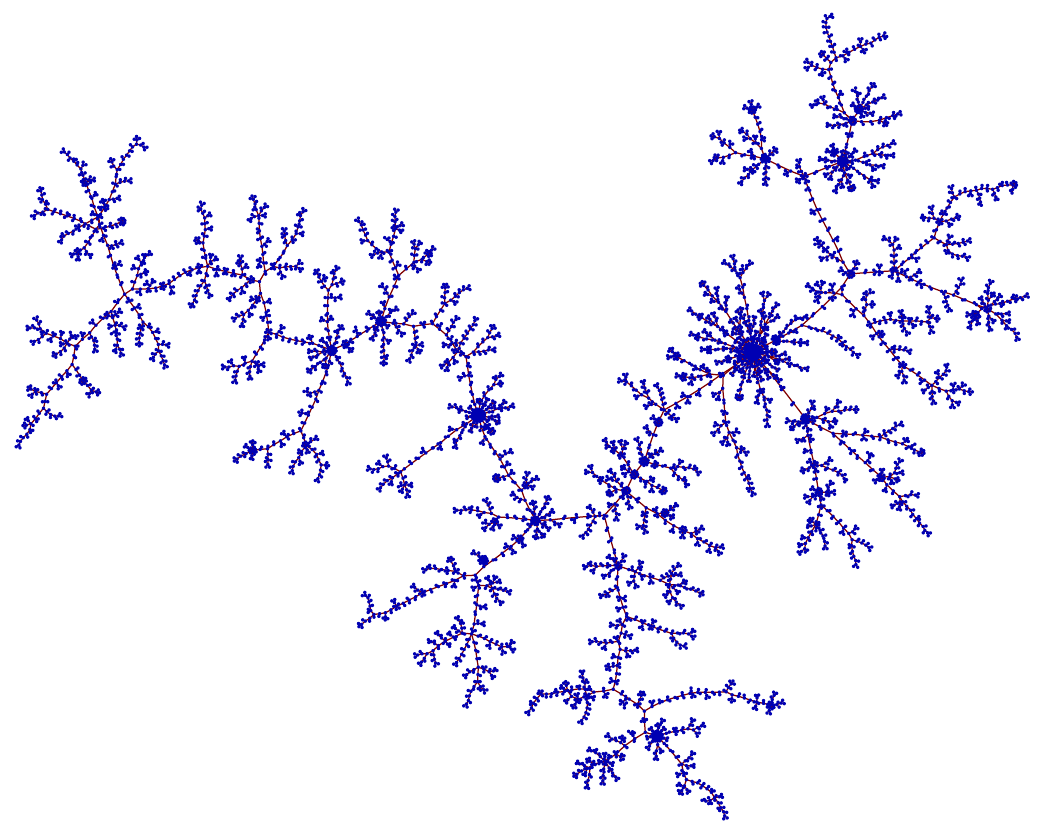}
    \includegraphics[height=4cm]{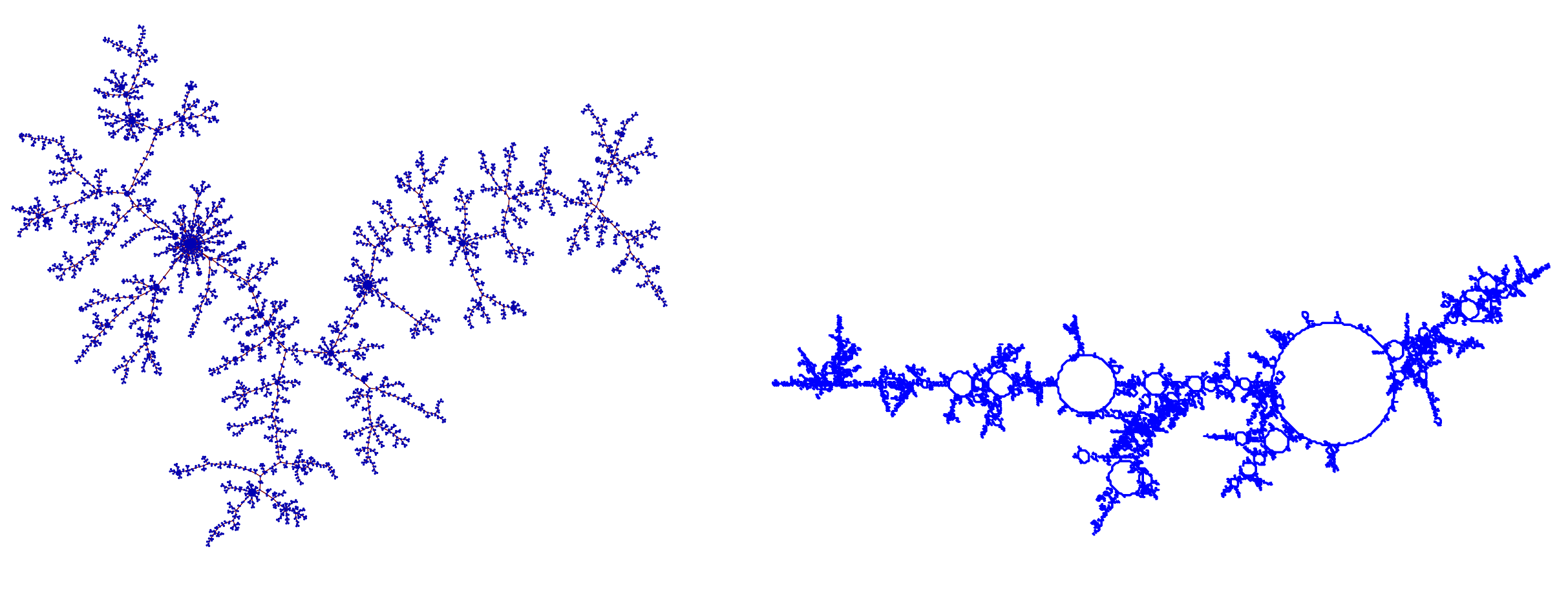}
 \caption{ \label{fig:frerots} An $ \alpha={3}/{2}$ stable tree, and its {associated}
 looptree {$\mathscr{L}_{3/2}$}, embedded non isometrically and non properly in the plane.}
 \end{center}
 \end{figure}
 
It has been proved in \cite{CK13} that the Hausdorff dimension of $ \mathscr{L}_{ \alpha}$ is almost surely equal to $\alpha$. Furthermore, the stable looptrees can be seen as random metric spaces interpolating between the unit length circle $ \mathcal{C}_{1}:=\frac{1}{2\pi} \cdot \mathbb{S}_{1}$ and  Aldous' Brownian CRT  \cite{Ald91a} (which we view here as the tree $ \mathcal{T}_{ \mathbf{e}}$ coded by a normalized Brownian excursion $ \mathbf{e}$, see \cite{LG05}). We are now in position to describe the possible scaling limits of the boundary of percolation clusters in the UIPT.  For fixed $a \in (0,1)$, let $ \partial \mathcal{H}_{a}^{\circ}(n)$ be the boundary of the white hull of the origin conditioned on the event that $ \mathcal{H}^{\circ}_{a}$ is finite and that the perimeter of $ \partial  \mathcal{H}^{\circ}_{a}$ is $n$. We view $  \partial \mathcal{H}_{a}^{\circ}(n)$ as a compact metric space by endowing its vertices with the graph distance (every edge has unit length).

 \bigskip

 \begin{theorem}[Scaling limits for $\partial \mathcal{H}^{ \circ}_{a}$ when $ \# \mathcal{H}_{a}^\circ < \infty$] \label{thm:scalingperco} For every $a \in (0,1)$, there exists a positive constant $C_{a}$ such that  the following convergences hold in distribution for the Gromov--Hausdorff topology:
$$\begin{array}{cllcl} (i) &  \mbox{when }1/2<a<1,&   n^{-1} \cdot \partial \mathcal{H}_a^{\circ}(n) &\xrightarrow[n\to\infty]{(d)}&  C_{a} \cdot \mathcal{C}_{1}, \\ \ \\
(ii) & \mbox{when }a=a_{c}=1/2,&  n^{-2/3} \cdot\partial \mathcal{H}_a^{\circ}(n)  &\xrightarrow[n\to\infty]{(d)}&  3^{1/3}  \cdot \mathscr{L}_{3/2}, \\ \ \\ 
(iii) & \mbox{when }0<a<1/2, &  n^{-1/2} \cdot \partial \mathcal{H}_a^{\circ}(n) &\xrightarrow[n\to\infty]{(d)}& C_a\cdot  \mathcal{T}_{ \mathbf{e}}.  \end{array}$$
See Theorem \ref{thm:nearcrit} below for more details about the constants $C_{a}$.
\end{theorem}

Although Theorem \ref{thm:scalingperco} does not imply that $1/2$ is the critical threshold for percolation on the UIPT (as shown in \cite{Ang03}), it is  a compelling evidence for it.
Let us give a heuristic justification for the three limiting compact metric spaces appearing in the statement of this theorem. Imagine that we condition the cluster of the origin to  be finite and have a very large, but finite, boundary. In the supercritical regime $(i)$, as soon as  the cluster grows arms it is likely to become infinite, hence the easiest way to stay finite is to look like a loop. On the contrary, in the subcritical regime $(iii)$, having a large boundary costs a lot, so  the cluster adopts the shape which maximizes its boundary length for fixed size: the tree. In the critical case $(ii)$, these effects are balanced and a fractal object emerges: not quite a loop, nor a tree, but a looptree!

The proof of Theorem \ref{thm:scalingperco} gives the expression of $C_{a}$ in terms of certain quantities involving Galton--Watson trees. This allows us to obtain the following near-critical scaling behavior.

\begin{theorem}[Near-critical scaling constants] \label{thm:nearcrit}The constants $C_{a}$ satisfy the following near-critical asymptotic behavior:
$$C_{a}  \underset{ a \downarrow 1/2}{\sim} \frac{ 2}{\sqrt{3}} \cdot \left(a-\frac{1}{2} \right) \qquad \textrm{and} \qquad C_{a}\underset{a \uparrow 1/2}{ \sim}  \displaystyle  \frac{3^ {3/4}}{8} \cdot  \left( \frac{1}{2}-a \right)^{-1/2}.$$
\end{theorem}
 
See \eqref{eq:Ca} below for the exact expression of $C_{a}$. 

\bigskip

Let us mention that the exponents appearing in the previous theorems are expected to be universal (see Section \ref{sec:universality} for analogous results for type II triangulations). Finally, we believe that our techniques may be extended to prove that the stable looptrees $(\mathscr{L}_{\alpha} : \alpha \in (1,2))$ give the scaling limits of the outer boundary of  clusters of suitable statistical mechanics models on random planar triangulations, see Section \ref{sec:conjectures}.

\paragraph{Strategy and organization of the paper.} In Section 2, we explain how to decompose a percolated triangulation into a white hull, a black hull and a necklace by means of a surgery along a percolation interface. In Section 3, we study the law of the white hull by using the so-called Boltzmann measure with exposure and its relation with a certain Galton--Watson tree. In Section 4, we prove our main results by carrying out explicit calculations. Finally, Section 5 is devoted to  comments, extensions and conjectures. 

\bigskip

\noindent \textbf{Acknowledgments: } The first author is indebted to Olivier Bernardi and Grégory Miermont for many useful discussions concerning percolation on random maps. 

\section{Decomposition of percolated triangulations} \label{sec:perco}
In this section, we explain how to decompose certain percolated triangulations into a pair of two hulls glued together by a so-called necklace of triangles. This sort of decomposition was first considered by Borot, Bouttier \& Guitter \cite{BBG11,BBG12}. We then associate a natural tree structure to a triangulation with boundary. The crucial feature is that when considering percolation on the UIPT, the random tree coding the hull turns out to be related to a Galton--Watson tree with an explicit offspring distribution in the domain of attraction of a stable law of index $3/2$.

\subsection{Percolated triangulations} \label{sec:percotrig}
A planar map is a proper embedding of a finite connected graph in the two-dimensional sphere, considered up to orientation-preserving homeomorphisms. The faces of the map are the connected components of the complement of the edges, and the degree of a face is the number of edges that are incident to it, with the convention that if both sides of an edge are incident to the
same face, this edge is counted twice. 

As usual in combinatorics, we will only consider \emph{rooted} maps that are maps with a distinguished oriented root edge. If $m$ is a planar map, we will denote by $\mathrm{V}(m)$, $ \mathrm{E}(m)$ and $ \mathrm{F}(m)$ respectively the sets of vertices, edges and faces of $m$.

A triangulation is a (rooted) planar map whose faces are all triangles, i.e.\,have degree three. Self-loops and multiples edges are allowed.  A triangulation with \emph{boundary} $T$  is a planar map whose faces are triangles except the face adjacent on the right of the root edge called the external face, which can be of arbitrary degree. The size $|T|$ of $T$ is its total number of vertices. The boundary $ \partial T$ is the graph made of the vertices and edges adjacent to the external face of $T$ and its perimeter $\# \partial T$ is the degree of the external face. The boundary is \emph{simple} if the number of vertices of $ \partial {T}$ is equal to its perimeter, or equivalently if $ \partial T$ is a discrete cycle.  In the following, we denote by $ \mathbb{T}_{n,p}$ the set of all triangulations with boundary of perimeter $p$ having $n$ vertices in total. For reasons that will appear later, the set $ \mathbb{T}_{2,2}$ is made of the ``triangulation'' made of a single edge. By splitting the root edge of a triangulation into two new edges seen as part of a new triangle (by adding a loop) it follows that $ \#\mathbb{T}_{n,1}$ is also the number of rooted triangulations with $n$ vertices in total.   When working with the UIPT we will also allow  triangulations to be infinite, see \cite{CMM12} for background (in the quadrangular case).

A \emph {percolated triangulation} is by definition a triangulation $T$ with a coloring of its vertices in black or white. We say that the percolation is \emph{nice} if the root edge joins a white vertex to a black vertex (which we write $ \circ \rightarrow \bullet$).
Note that this forces a percolation interface to go through the root edge, and that the latter cannot be a self-loop. The origin of the root edge is called the \emph{white origin} and its target is called the \emph{black origin}.

\medskip

\begin{center}\emph{In the following, we always assume that the percolation is nice.}
\end{center}

\subsection{Necklace surgery}
\label{sec:necklace}
 
In this section, we assume in addition that the percolation interface going through the root edge is finite, see Fig.~\ref{fig:decoupage}.

\paragraph{The white and black hulls.} Let $T$ be a nicely percolated triangulation (possibly infinite). The white cluster is by definition the submap consisting of all the edges (together with their extremities) whose endpoints are in the same white connected component as the white origin. The complement of this cluster consists of connected components.
The white hull $ \mathcal{H}^\circ$ is the union of the white cluster and of all the latter connected components, except the one containing the black origin (see Fig.~\ref{fig:decoupage}). Hence $ \mathcal{H}^{\circ}$ is a triangulation with a (non necessarily simple) boundary, and is by convention rooted at the edge whose origin is the white origin (with the external face lying on its right). Note that by definition all the vertices on the boundary of $ \mathcal{H}^\circ$ are white.

We similarly define the black cluster as the submap consisting of all the edges (together with their extremities) whose endpoints are in the same connected component as the black origin. The black hull $ \mathcal{H}^\bullet$ is {similarly} obtained by filling-in the holes of the black cluster except the one containing the white origin. The map $ \mathcal{H}^\bullet$ is thus a triangulation with a black boundary which is rooted at the edge whose origin is the black  origin. Recall that the perimeters of $ \mathcal{H}^\circ$ and $ \mathcal{H}^\bullet$ are respectively denoted by $\# \partial \mathcal{H}^\circ$ and $\# \partial \mathcal{H}^\bullet$. These quantities are finite by the assumption made in the beginning of this section. However, $|\mathcal{H}^\bullet|$ and $|\mathcal{H}^\circ|$ may be infinite.

\paragraph{Surgery.} Imagine that using a pair of scissors, we separate the  two hulls $ \mathcal{H}^\circ$ and $ \mathcal{H}^\bullet$ by cutting along their boundaries. After doing so, we are left with the two hulls $\mathcal{H}^\circ$ and $ \mathcal{H}^\bullet$ which are now separated, together with the map that was stuck in-between $ \mathcal{H}^\bullet$ and $ \mathcal{H}^\circ$, which is called the \emph{necklace} \cite{BBG11,BBG12}. During this operation, we duplicate in the necklace the vertices which are pinch-points in the boundary of $ \mathcal{H}^\circ$ or $ \mathcal{H}^\bullet$, see Fig.~\ref{fig:decoupage}.
 
 \begin{figure}[!h]
 \begin{center}
 \includegraphics[width=17cm]{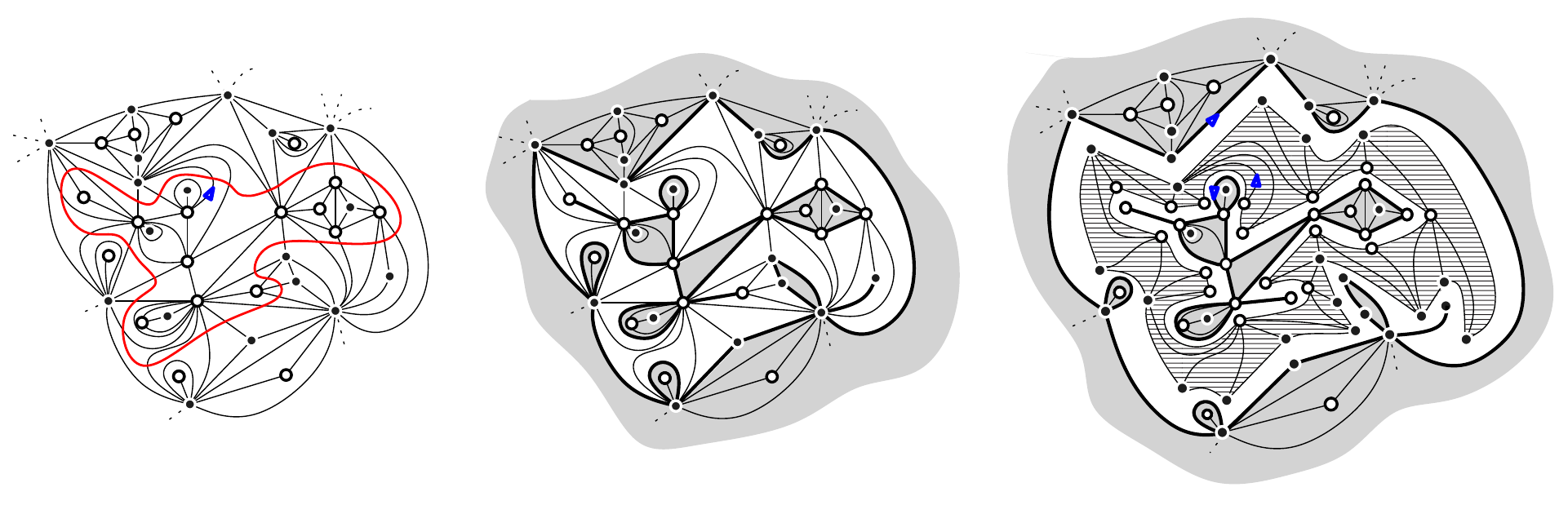}
 \caption{ \label{fig:decoupage} From left to right, a nicely percolated triangulation, the two hulls $ \mathcal{H}^\bullet$ and $ \mathcal{H}^\circ$ (in light gray) as well as the creation of the necklace (dashed on the right) after the surgical operation.}
 \end{center}
 \end{figure}

If $n,m \geq 0$  are integers such that $ m+n > 0$, a \emph{$(n,m)$-necklace} is  by definition a triangulation with two simple boundaries, a ``white'' one of perimeter $n$ and a ``black'' one of perimeter $m$, such that every vertex belongs to one of these two boundaries and such that every triangle has at least one  vertex on each boundary, rooted along an edge joining a white vertex to a black one. The root of the necklace obtained from a nicely percolated triangulation $T$ is the root of $T$. It is an easy exercise to show that for $n,m \geq 0$, $$ \#\{ (n,m) \mbox{-necklaces}\} = { {n+m \choose n}}.$$

 \begin{figure}[!h]
 \begin{center}
\includegraphics[width=0.5 \linewidth]{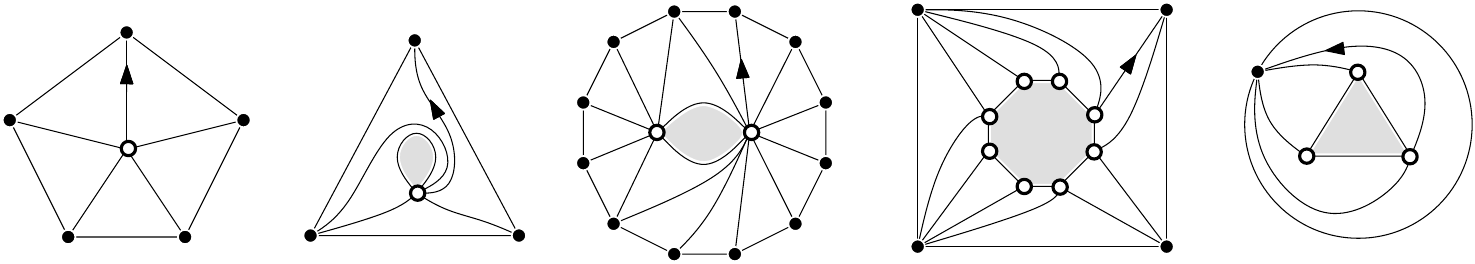}
 \caption{ Examples of $0$--$5$, $1$--$5$, $2$--$12$, $8$--$4$ and $3$--$1$ necklaces.}
 \end{center}
 \end{figure}

It is plain that the last decomposition is invertible, in other words the following result holds: 
\begin{proposition}[Necklace surgery] Every nicely percolated triangulation, such that the interface going through the root edge is finite, can be unambiguously decomposed into a pair of two triangulations with boundary $( \mathcal{H}^{\circ}, \mathcal{H}^{\bullet})$ forming the white and black hulls glued together along a 
$ (\# \partial \mathcal{H}^{\circ} ,\# \partial \mathcal{H}^\bullet) \mbox{-necklace}.$
\end{proposition}

In the next subsection, we further decompose a hull according to the tree structure provided by its $2$-connected components.

\subsection{Tree representation of triangulation with boundary} \label{sec:scoop}

We denote by $ \mathcal{T}$ the set of all plane (rooted and oriented) trees, see \cite{LG05,Nev86} for the formalism. In the following,  \emph{tree} will always mean plane tree.  We will view each vertex of a tree $\tau$ as an individual of
a population whose $\tau$ is the genealogical tree.  The vertex $\varnothing$ is the ancestor of this population and {is} called the root. The degree of a vertex $u \in \tau$ is denoted by $ \textrm{deg}(u)$ and its number of children is denoted by $k_{u}$. The size of $ \tau$ is by definition the total number of vertices and will be denoted by $| \tau|$ and $ \textsf{H}( \tau)$ is the height of the tree, that is, its maximal generation.

\bigskip

 We denote by $ \mathbb{T}^B$ the set of all triangulations with boundary and by $ \mathbb{T}^{S}$ the set of all triangulations with simple boundary {(also called simple triangulations in the sequel)}. Let $T$ be a triangulation with boundary. We recall that the perimeter $ \# \partial T$ of $T$ is the number of half-edges on its boundary. We define the set $ \mathcal{E}(T)$ of all \emph{exterior vertices} of $T$ as the set of all the vertices on the boundary of $T$, and $ \mathcal{I}(T)$ as the complement of $ \mathcal{E}(T)$ in $ \mathrm{V}(T)$, which are the so-called \emph{inner vertices} of $T$. Note that $ \#\mathcal{E}(T) = \# \partial T$ when $T$ has a simple boundary, and that $ \#\mathcal{E}(T) < \# \partial \mathsf{T}$ otherwise. If $T$ is not a simple triangulation, then $T$ can be decomposed into $ \# \partial T - \#\mathcal{E}(T) +1$ different simple triangulations, attached by some pinch-points, see Fig.~\ref{scoop}.

Imagine that we scoop out the interior of all these simple triangulation components and duplicate each edge whose sides both belong to the external face into two ``parallel'' edges (see Fig.~\ref{scoop}). We thus obtain a collection of cycles glued together, which we call the \emph{scooped-out triangulation} and denote it by $\mathsf{Scoop}(T)$. Note that  $\mathsf{Scoop}(T)$ may differ from the boundary $\partial T$ of $T$ only because some  edges may have been duplicated, see Fig.~\ref{scoop}. Note however that the underlying metric spaces are identical. \begin{figure}[!h]
 \begin{center}
 \includegraphics[width=0.8 \linewidth]{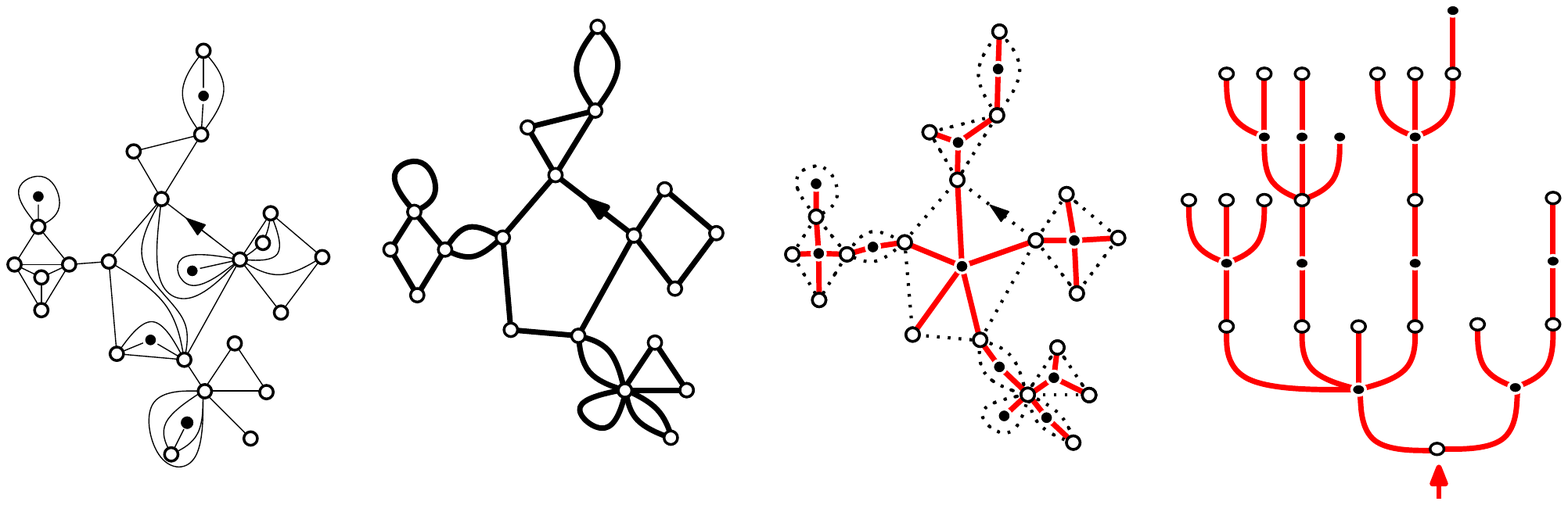}
 \caption{From left to right, a triangulation with boundary, its scooped-out triangulation and the tree structure underneath. \label{scoop}}
 \end{center}
 \end{figure}

 The scooped-out triangulation $\mathsf{Scoop}(T)$ can naturally be represented as a  tree. More precisely, with $\mathsf{Scoop}(T)$ we associate  a tree with two types of vertices, white and black, as follows.
 Inside each cycle, add a new black vertex which is connected to all the white vertices belonging to this cycle.
 The resulting tree is denoted by $ \mathrm{Tree}(T)$ and is rooted at the corner adjacent to the target of the root edge of  $T$ (see Fig.~\ref{scoop}).
 By construction, $ \mathrm{Tree}(T)$ is  a plane tree such that all the vertices at even (resp.~odd) height are white (resp.~black). 
 If $t$ is a plane tree, let $ \bullet(t)$ (resp.~$ \circ(t)$) be the set of all vertices at odd (resp.~even) height.
 If $t = \mathrm{Tree}(T)$, then the vertices belonging to $ \circ(t)$ correspond to the exposed vertices of $T$ and the following relations are easy to check:
  \begin{eqnarray}  \label{eq:nombreexpo}\#\circ(t) &=& \# \mathcal{E}(T) \\ \nonumber \\ 
   \label{eq:bn} |t| = \#\bullet( t)+\#\circ(t)&=& \# \partial T+1\\ \nonumber  \\
   \sum_{ u \in \bullet(t)} \deg (u) &=&  \# \partial  \mathsf{T} \label{eq:sumdegblack} \\
   \sum_{u \in \circ(t)} (1+ k_{u})&=&|t|. \label{eq:sumwhite}
   \end{eqnarray} This scooping-out procedure is a bijection between the set of all triangulations with boundary and the set of all plane trees having at least two vertices together with a finite sequence $(T_{u}, u \in  \bullet(t))$ of triangulations with \emph{simple} boundary such that $\#\partial T_{u} = \mathrm{deg}(u)$ (which correspond to the triangulations inside each cycle).
Recall that $\mathbb{T}_{2,2}$,  the set of all triangulations with simple boundary of perimeter $2$ and no internal vertices, is by convention composed by a degenerate triangulation made of a single edge : We use this triangulation to close a double edge into a single one, see Fig.~\ref{scoop}.
 
\subsection{Enumerative results}
 
In the previous section, we have explained how to decompose a triangulation with boundary into a tree of   triangulations with simple boundary. We now present some useful enumerative results on triangulations with simple boundary. \medskip

 We denote by  $ {W}$ the generating function of triangulations with simple boundary having weight $x$ per \emph{inner} vertex and $y$ per edge on the boundary, that is 
\begin{eqnarray*}{W}(x,y) & :=& \sum_{T \in  \mathbb{T}^{S}} x^{\# \mathcal{I}(T)} y^{\#\partial T} =  
 yx + y^2+\cdots .   \end{eqnarray*}
 Note that the contribution of the ``edge triangulation'' is  $y^2$ in the previous sum. We also let ${w}_{n,p} = [x^n][y^p] {W}$ be the number of triangulations with simple boundary of perimeter $p \geq 1$ and $n$ \emph{internal} vertices. Following Tutte \cite{Tut62}, Krikun \cite{Kri07} calculated the generating function $ {W}(x,y)$. In particular, for $ y \in [0,1/12]$ the radius of convergence of ${W}$ as a function of $x$ is $$r_c := \frac{1}{\sqrt{432}}=  \frac{1}{12 \sqrt{3}},$$ and an explicit formula for $w_{n,p}$ can be found in \cite{Kri07} (Krikun uses the number of edges as size parameter; to translate his formulas recall that if a triangulation of the $p$-gon has $n$ inner vertices then by Euler's formula it has $3n+2p-3$ edges). We will not need its exact expression,  but we will heavily rely on the following asymptotic estimates: 
    \begin{eqnarray} \label{eq:equiwnp} {w}_{n,p} \quad \underset{n\to \infty}{\sim} \quad C_{p}  \cdot n^{-5/2} \ r_{c}^{-n}, \qquad C_{p} = \frac{3^{p-2}   \cdot p \cdot (2 p)!}{4 \sqrt{2 \pi} \cdot  (p!)^2}  \quad \underset{p \to \infty}{\sim} \quad \frac{1}{36\pi \sqrt{2}} \cdot \sqrt{p} \ 12^p.  \end{eqnarray}
In particular, note that  the number of triangulations with $n$ vertices is \begin{eqnarray} \label{eq:equiwnp2} \# \mathbb{T}_{n,1} =w_{n-1,1}  \quad\mathop{ \sim }_{n \rightarrow \infty} \quad  \frac{1}{72 \sqrt{6 \pi }} \cdot n^{-5/2} \ r_{c}^{-n}. \end {eqnarray} We will also use the explicit expression of $  {{W}}(r_{c},y)$: 
  \begin{eqnarray} \label{eq:Wtildeexact} {{W}}(r_{c},y) &=& \frac{y}{2}+ \frac{ (1- 12 y)^ {3/2}-1}{24 \sqrt {3}}. \end{eqnarray}
This expression can be obtained from \cite[(4)]{Kri07} after a change of variables (with the notation of Krikun we have formally $W(x^3,yx^2) = x^3 U_{0}(x,y)$). For every integer $k \geq 1$, we also introduce 
 \begin{eqnarray} \label{eq:defqk} q_{k} \quad :=  \quad \frac{1}{12^{k}} \cdot  [y^k] { {W}}(r_{c},y) \quad =  \quad \frac{1}{12^k} \cdot \sum_{n=0}^\infty  w_{n,k} r_{c}^{n}. \end{eqnarray} Standard singularity analysis shows that  \begin{eqnarray} \label{eq:asymqk} q_{k}  \quad \underset{k \to \infty}{\sim} \quad \frac{1}{32 \sqrt{3 \pi}} \cdot k^{-5/2}. \end{eqnarray} In particular, note that the series $ \sum_{k \geq 1} q_{k}$ is convergent.

\section{Boltzmann triangulations with exposure and GW trees}  \label{sec:BTWE}
This section is devoted to the study of the tree structure of a random triangulation with boundary distributed according to the Boltzmann measure with exposure defined below. This measure will naturally arise in Proposition \ref{prop:loiH} when considering the hulls in a Bernoulli site percolation of the UIPT.

\begin{definition}[Boltzmann measure with exposure on triangulations] For every $a \in (0,1)$, we  introduce a measure $ Q_{a}$  on the set of all triangulations with (general) boundary, called the critical Boltzmann measure with exposure $a$, defined by 
\begin{eqnarray} \label{def:critical} Q_{a}( \mathsf{T}) &= & r_c^{\# \mathrm{V}(T)} \ {12^{-\# \partial T}} \ a^{\# \mathcal{E}(T)}, \qquad \qquad  \forall \ T  \in \mathbb{T}^B.  \end{eqnarray} 
\end{definition}
Note that in this definition, $r_{c}$ is elevated to the power $ \# \mathrm{V}(T)$ and not to the power $ \# \mathcal{I}(T)$ as in $W$.
  Our goal is now to describe the ``law'' of the tree of components of a triangulation under the measure $ Q_{a}$.

 \subsection{A two-type Galton--Watson tree}
 \label{section:twotype}
Given two probability measures $\mu^{\circ}$ and $ \mu^{ \bullet}$ on $\{0,1,2,3, \ldots\}$, we consider a two-type Galton--Watson tree where every vertex at even (resp.~odd) height has a number of children distributed according to $\mu^{\circ}$ (resp.~$\mu^{\bullet}$), all independently of each other. Specifically, using the notation $k_{u}$ for the number of children of a vertex $u$ in a plane tree, its law, denoted by $\mathsf{GW}_{\mu^{\circ},\mu^{\bullet}}$, is characterized by the following formula:
$$\mathsf{GW}_{\mu^{\circ},\mu^{\bullet}}(t)=  \prod_{ u \in \bullet(t)} \mu^{\bullet}( k_{u}) \prod_{ u \in \circ(t)} \mu^{ \circ}( k_{u}), \qquad \qquad \forall \ t \in \mathcal{T} .$$
Recall that $a \in (0,1)$.  To simplify notation, set $\gamma = \sqrt{3}-1, \xi= \gamma/( \gamma+2a)$ and define  two probability distributions  $\mu^{\bullet}$ and $\mu^{\circ}_{a}$  on $\{0,1,2,3, \ldots\}$ by
 $$\mu^{\bullet}(j) = \frac{q_{j+1}}{Z_{\bullet}}, \qquad   \mu^{\circ}_{a}(j) = (1-\xi)\xi^j  \qquad (j \geq 0),$$
 where  $Z_{ \bullet}$ is a normalizing constant. Using \eqref{eq:Wtildeexact}, simple computations show that $Z_{\bullet} = \gamma {r_{c}}/{2}$. The following proposition is the key of this work:
  
\begin{proposition} \label{prop:gwQ} For every $a \in (0,1)$ and for every plane tree $t$ such that $|t| \geq 1$, we have
 \begin{eqnarray*}Q_{a} \left(\{ \mathsf{T} \in  \mathbb{T}^B ; \,  \mathsf{Tree}(T)=t\} \right) &=& \left(\frac{r_{c}(2a+\gamma)}{2}\right)^{|t|} \cdot \mathsf{GW}_{\mu^{\circ}_{a},\mu^{\bullet}}(t).
 \end{eqnarray*} 
 \end{proposition}

\proof 
Fix $t \in \mathcal{T}$ with $|t| \geq 1$. 
 Using \eqref{def:critical}, the scoop decomposition of Section \ref{sec:scoop} (and its consequence \eqref{eq:nombreexpo}), \eqref{eq:sumdegblack}, the definition of the $q_{k}$ in \eqref{eq:defqk}, then \eqref{eq:bn} and finally \eqref{eq:sumwhite}, we get that
  \begin{eqnarray*}
Q_{a}(\{ \mathsf{T} : \mathsf{Tree}(T)=t\}) & = &\frac{1}{12^{\# \partial  \mathsf{T}}} \prod_{ u \in  \circ(t)}^{} ar_{c}  \sum_{n_{1} , \ldots, n_{\# \bullet(t)}  \geq 0} \hspace{0.2cm}\prod_{ u \in \bullet(t)} r_{c}^{n_{i}} w_{n_{i},  \mathrm{deg}(u)} \\
 &\underset{ \eqref{eq:sumdegblack}, \eqref{eq:defqk}}{=}&  \prod_{u\in \circ(t)}a r_{c}   \prod_{ u \in \bullet(t)} q_{k_{u}+1}\\
&\underset{\eqref{eq:bn}}{=}&  Z_{ \bullet}^{\# \bullet(t) + \# \circ(t)} \prod_{u\in \circ(t)}\frac{ar_{c}}{Z_{\bullet}} \prod_{ u \in \bullet(t)} \frac{q_{k_{u}+1}}{Z_{ \bullet}}\\
&\underset{\eqref{eq:sumwhite}}{=}&\left(\frac{Z_{ \bullet}}{\xi}\right)^{\# \bullet(t) + \# \circ(t)}  \prod_{u\in \circ(t)}\frac{a r_{c}}{Z_{\bullet}}\xi^{k_{u}+1}   \prod_{ u \in \bullet(t)} \frac{q_{k_{u}+1}}{Z_{ \bullet}}. \end{eqnarray*}
Since $ \xi a r_{c}/Z_{ \bullet}=1- \xi$, this completes the proof.  \endproof

 \begin{rek}\label{rem:crit}A simple computation shows that the mean of $\mu^{\bullet}$ is equal to $m^{ \bullet} =  1/\gamma$ and that the mean of $\mu^{\circ}_{a}$ is   $m^{ \circ}_{a} = \gamma/(2a)$. In particular $m^{\bullet}m^{ \circ}_{a}=1/(2a)$, so that the  two-type Galton--Watson tree is critical if and only if $a=1/2$.
 \end{rek}
 
The following proposition will be useful when we will deal with the UIPT. Recall that $ \mathbb{T}_{n,p}$ is the set of all triangulations with general boundary of perimeter $p$ having $n$ vertices \emph{in total}.

\begin{proposition} \label{prop:phi} For every fixed $p \geq 1$, we have
$ \displaystyle Q_{a} ( \mathbb{T}_{n,p})  \sim   K_a(p) \cdot n^{-5/2}$ as $n \to \infty$,   with
 \begin{eqnarray*}
 K_{a}(p) &=& \left(  \frac{r_{c}(\gamma + 2a)}{2} \right)^{p+1} \mathsf{GW}_{ \mu^{\circ}_a,\mu^{\bullet}} \left[  \sum_{u \in \bullet(\tau)}  \phi(k_{u}) \mathbbm{1}_{|\tau|=p+1}\right],  \end{eqnarray*} where
 $$\phi(k) \quad = \quad	  \frac{C_{k+1}}{12^{k+1}q_{k+1}}  \quad\mathop{\sim}_{k \rightarrow \infty} \quad   \frac{4}{9}  \sqrt{ \frac{6}{ \pi}} \cdot k^3.$$\end{proposition}

\proof By \eqref{eq:bn}, if $T \in \mathbb{T}_{n,p}$ and $t = \mathrm{Tree}(T)$, then $ 1 \leq \# \circ(t) \leq p$. Using the scoop-out decomposition, we can  thus write 
 \begin{eqnarray} \label{eq:avant} Q_{a}( \mathbb{T}_{n,p}) &=& \frac{1}{12^p} \cdot \sum_{\begin{subarray}{c} t \in \mathcal{T} \\ |t|=p+1 \end{subarray}} \prod_{ \circ(t)}^{} ar_{c}  \sum_{n_{1} + ... + n_{\# \bullet(t)} = n - \# \circ(t)} \hspace{0.2cm}\prod_{ u \in \bullet(t)} r_{c}^{n_{i}} w_{n_{i},  \mathrm{deg}(u)}.\end{eqnarray}
As $n \to \infty$, a standard phenomenon occurs: in the second sum appearing in \eqref{eq:avant}, the only terms $n_{1}, \ldots , n_{ \#\bullet(t)}$ that have a contribution in the limit are those where  one term is of order $n$ whereas all the others remain small. More precisely, we use the following lemma whose proof is similar to that of \cite[Lemma 2.5]{AS03} and is left to the reader: 
 \begin{lemma} \label{lem:cfg}Fix an integer $k \geq 0$ and $ \beta>1$. For every $i \in \{1,2, \ldots , k\}$, let $(a^{(i)}_{n} ; n \geq 0)$ be a sequence of positive numbers such that $a^{(i)}_{n} \sim C_{i}  \cdot n^{-\beta}$ as $ n \rightarrow \infty$. Then 
 $$ \lim_{n \to \infty}  n^{\beta} \sum_{n_{1}+ ... + n_{k}=n} \hspace{0.3cm} \prod_{i=1}^k a^{(i)}_{n_{i}} = \sum_{i=1}^k C_{i} \prod_{ j \ne i} \sum_{n=0}^\infty a^{(j)}_{n}.$$\end{lemma}

Multiplying both sides of \eqref{eq:avant} by $n^{5/2}$,  by \eqref{eq:equiwnp} we are in position to apply Lemma \ref{lem:cfg}  together with   the definition of $q_{k}$ \eqref{eq:defqk} to get:
 \begin{eqnarray*}  \lim_{n \to \infty}n^{5/2} Q_{a}( \mathbb{T}_{n,p}) & = & \frac{1}{12^p}\sum_{\begin{subarray}{c} t \in \mathcal{T} \\ |t|=p+1 \end{subarray}} \prod_{ \circ(t)}^{} ar_{c}  \sum_{ u \in \bullet(t)} C_{ \mathrm{deg}(u)}  \prod_{  \begin{subarray}{c} v \in \bullet(t) \\ v \ne u \end{subarray}} 12^{ \mathrm{deg}(v)} q_{ \mathrm{deg}(v)}\\
 &\underset{ \eqref{eq:sumdegblack}}{=}&\sum_{\begin{subarray}{c} t \in \mathcal{T} \\ |t|=p+1 \end{subarray}}  \prod_{   \circ(t)}^{} ar_{c}\sum_{ u \in \bullet(t)} \frac{C_{ \mathrm{deg}(u)}}{12^{ \mathrm{deg}(u)}}  \prod_{  \begin{subarray}{c} v \in \bullet(t) \\ v \ne u \end{subarray}} q_{ \mathrm{deg}(v)}\\
 &=& \sum_{\begin{subarray}{c} t \in \mathcal{T} \\ |t|=p+1 \end{subarray}}  \prod_{   \circ(t)}^{} ar_{c}  \prod_{  \begin{subarray}{c} v \in \bullet(t)  \end{subarray}} q_{ \mathrm{deg}(v)} \sum_{ u \in \bullet(t)} \frac{C_{ \mathrm{deg}(u)}}{12^{ \mathrm{deg}(u)} q_{ \mathrm{deg(u)}}}.\end{eqnarray*} Performing the same manipulations as in the proof of Proposition \ref{prop:gwQ},  the last display is equal to
 \begin{eqnarray*} &=& \left(  \frac{r_{c}(\gamma + 2a)}{2} \right)^{|t|} \sum_{\begin{subarray}{c} t \in \mathcal{T} \\ |t|=p+1 \end{subarray}}  \prod_{ u \in \circ(t)}^{} \mu^{\circ}_a( k_{u}) \prod_{ u \in \bullet(t)} \mu^{ \bullet}(k_{u}) \sum_{ u \in \bullet(t)} \frac{ C_{k_{u}+1}}{12^{k_{u}+1}  { q_{k_{u}+1}}} \\
 &=& \left(  \frac{r_{c}(\gamma + 2a)}{2} \right)^{p+1} \mathsf{GW}_{ \mu^{\circ}_a,\mu^{\bullet}} \left[  \sum_{u \in \bullet(\tau)}  \phi(k_{u}) \mathbbm{1}_{|\tau|=p+1}\right], \end{eqnarray*} where $\phi(k) = C_{k+1}/(12^{k+1}q_{k+1})$ is asymptotically equivalent to $ \frac{4}{9}  \sqrt{ \frac{6}{ \pi}} \cdot k^3$ as $k \rightarrow \infty$ by \eqref{eq:asymqk} and \eqref{eq:equiwnp}.  \endproof 
 
 We conclude this section by giving the asymptotic behavior of the expectation appearing in the definition of $ K_{a}(p)$  as $ p \rightarrow \infty$, in the critical case $a =1/2$: 
\begin{equation}
\label{eq:constante}\mathsf{GW}_{\mu^{\circ}_{\text{\sfrac{1}{2}}},\mu^{\bullet}} \left[ { \sum_{u \in \bullet(\tau)} \phi(k_{u})  \mathbbm{1}_{| \tau|=p+1} } \right]  \quad\mathop{ \sim}_{p \rightarrow \infty} \quad   \frac{3^{1/6}}{  \Gamma(-2/3)^2 \cdot\sqrt{ 2\pi}} \cdot p^{1/3}.
\end{equation}
 The proof is postponed to the appendix (see Corollary \ref {cor:At}).

\subsection {Reduction to a one-type Galton--Watson tree}
\label{section:bijection}
We have seen in the last section that the ``law'' of the tree associated with a Boltzmann triangulation with exposure is closely related to a \emph{two-type} Galton--Watson tree. In order to study this random tree, we will use a bijection due to Janson \& Stefánsson \cite[Section 3]{JS12} which will map this two-type Galton--Watson tree to a standard \emph{one-type} Galton--Watson tree, thus enabling us to use the vast literature on  this subject.

We start by describing this bijection, denoted by $ \mathcal{G}$ (the interested reader is referred to \cite{JS12} for further details). First set $ \mathcal {G}( \tau)=  \{\varnothing\}$ if $ \tau=  \{\varnothing\}$ is composed of a single vertex. Now fix a tree $ \tau \neq  \{ \varnothing \}$. The tree $ \mathcal{G}( \tau)$ has the same vertices as $ \tau$, but the edges are different and are defined as follows. For every white vertex $u$ repeat the following operation : denote $u_{0}$ be the parent of $u$ (if $u \ne \varnothing$) and then list the children of $u$ in lexicographical order $u_{1},u_{2}, ... , u_{k}$. If $u \ne \varnothing$ draw the edge between $ u_{0}$ and $u_{1}$ and then edges between $u_{1}$ and $u_{2}$, ... , $u_{k-1}$ and $u_{k}$ and finally between $u_{k}$ and $u$. If $u$ is a white leaf this reduces to draw the edge between $u_{0}$ and $u$. 
One can check that the graph  $ \mathcal{G}( \tau)$ defined by this procedure is a tree. In addition, $\mathcal{G}( \tau)$ is rooted at the corner between  the root of $\tau$ and its first child (see Fig.~\ref {fig:bij}). 

\begin{figure}[!h]
 \begin{center}
\includegraphics[width= 0.9 \linewidth]{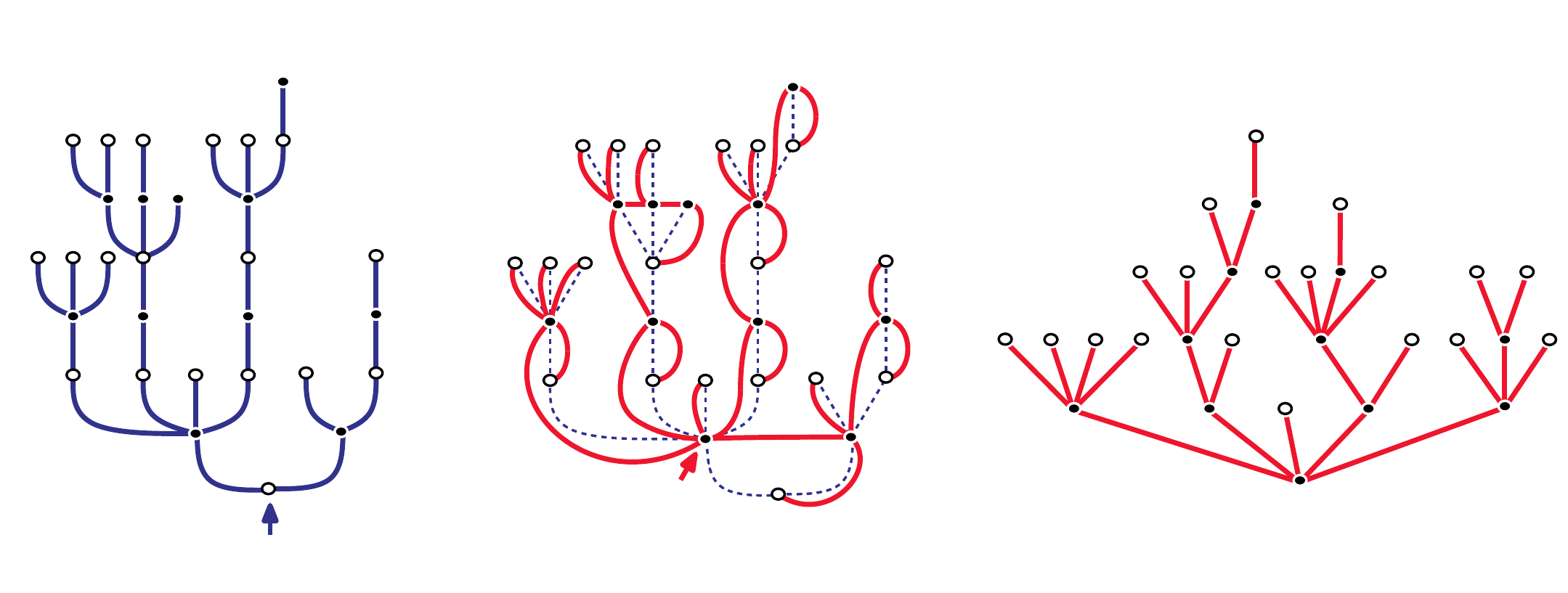}
 \caption{An example of a tree $ \tau$ (left) where vertices at even (resp.~odd) generation have been colored in white (resp.~black), and two representations of $ \mathcal{G}( \tau)$ (middle and right).
  \label {fig:bij}}
 \end{center}
 \end{figure}
 
This mapping thus has the property that every vertex at even generation is mapped to a leaf, and every vertex at odd generation with $k \geq 0$ children is mapped to a vertex with $k+1$ children. The following result is implicit in \cite[Appendix A]{JS12}, but for sake of completeness we give a   proof. \begin{proposition}[\cite{JS12}]\label {prop:onetype}Let $ \rho,\mu$ be two probability measures on $  \{0,1,2, \ldots\}$. Assume that $ \rho$ is a geometric distribution, i.e. there exists $ \lambda \in (0,1)$ such that $ \rho(i)= (1- \lambda) \lambda^{i}$ for $i \geq 0$. Then the image of $ \mathsf{GW}_ {\rho,\mu}$ under  $ \mathcal{G}$ is the Galton--Watson measure $ \mathsf{GW}_{ \nu}$, where $ \nu$ is defined by:
$$ \nu_{0}= 1- \lambda, \qquad \nu_{k}=  \lambda  \cdot \mu_{k-1}, \quad k \geq 1.$$
\end{proposition}

\begin {proof}
Fix a tree $t$. Color in white the vertices at even generation in $t$ and in black the other vertices. Recall that $\bullet( t)$ (resp.~$\circ(t)$) is the set of all black (resp.~white) vertices. Then, using the fact that $\# \bullet( t) + \# \circ(t) =  \sum_{u \in \circ(t)} (1+ k_{u})$, write \begin{eqnarray*}
\mathsf{GW}_ { \rho,\mu}( t) &=& \prod_{ u \in \bullet(t)} \mu_{k_{u}} \prod_{u\in \circ(t)} \lambda^{k_{u}+1} (1/ \lambda-1) \\
&=& \lambda^ { \#\bullet(t)+ \#\circ( t)} \prod_{ u \in \bullet(t)} \mu_{k_{u}} \prod_{u\in \circ(t)} (1/ \lambda-1) \\
&=& \prod_{u\in \circ(t)} ( 1- \lambda) \cdot \prod_{ u \in \bullet(t)} \lambda \mu_{k_{u}}.
\end{eqnarray*}
Since $ \mathcal{G}$ maps white vertices to leaves and black vertices with $k$ children to vertices with $k+1$ children, the last expression implies that for a tree $ \tau$:
$$\mathsf{GW}_ { \rho,\mu}( \mathcal{G}^{-1}( \tau))= \prod_{\begin{subarray}{c}u \in \tau\\  k_{u}=0 \end{subarray}} (1- \lambda) \prod_{\begin{subarray}{c}u \in \tau\\  k_{u}>0 \end{subarray}} \lambda \cdot\mu_{k_{u}-1}.$$
The conclusion follows.
\end {proof}

\paragraph{Application.}In virtue of Proposition \ref {prop:onetype} (applied with $ \mu=\mu^{\bullet}$ and $ \rho=\mu^{\circ}_{a}$), the image of $ \mathsf{GW}_{\mu^{\circ}_{a},\mu^{\bullet}}$ by $ \mathcal{G}$ is a standard Galton--Watson measure with offspring distribution $ \nu_{a}$ on $  \{0,1, \ldots\}$ defined by
$$ \nu_{a}(0)= \frac{ 2a}{ \gamma+2a}, \qquad \qquad \nu_{a}(k)=\frac{ 2}{ r_{c}(\gamma+2a)} q_{k} \quad (k \geq 1).$$
Using \eqref{eq:Wtildeexact}, it is a simple matter to check that the generating function of $ \nu_{a}$ is given by:
\begin{equation}
\label{eq:nua} F_a(z)=\sum_{i \geq 0} \nu_{a}(i) z^i=  \frac{2a-1 + \sqrt {3} z + (1-z)^{3/2}}{2a- 1+ \sqrt {3}}.
\end{equation}
In particular $F_a'(1)= (1+ 2(a-1/2)/ \sqrt {3})^{-1}$, so that $ \nu_{a}$ is critical if and only if $a=1/2$. When $a=1/2$, to simplify notation we write $ \nu= \nu_{\text{\sfrac{1}{2}}}$. Then note that:
\begin{equation}
\label{eq:nu}\nu(0)= \frac{ \sqrt {3}}{3} \textrm{ and }\nu(i)= 24 \cdot q_{i}  \, \,  (i \geq 1), \qquad \sum_{i \geq 0} \nu(i) z^i=z+ \frac{(1-z)^{3/2}}{ \sqrt {3}}, \qquad \nu(k)   \mathop{ \sim }_{k \rightarrow \infty}   \frac{ \sqrt{3}}{4 \sqrt{\pi}}  \cdot k^{-5/2}.
\end{equation}
In particular, this enables us to find the asymptotic behavior of the probability that our two-type Galton--Watson tree has total size $n$ (see Corollary \ref {cor:At} for this well-known fact):
\begin{equation}
\label{eq:equivtaille}\mathsf{GW}_{\mu^{\circ}_{\text{\sfrac{1}{2}}},\mu^{\bullet}} \left(|\tau|=n\right) = \GWnu{| \tau|=n} \quad\mathop{ \sim}_{n \rightarrow \infty} \quad  \frac{3^{1/3}}{|\Gamma(-2/3)|} \cdot n^{-5/3}.
\end{equation}

\begin{rek}The exponent $5/3$ in \eqref{eq:equivtaille} also appears (O.~Bernardi, personal communication) when analyzing the Boltzmann distribution with exposure using generating functions and methods from the theory of analytic combinatorics.
\end{rek}

\section{Study of the percolation hull}
With the tools developed in the previous sections, we can now proceed to the proofs of our main results. We start by identifying the law of the hulls of a nicely percolated UIPT (Proposition \ref{prop:loiH}), and then connect it with the Boltzmann measure with exposure introduced in Section \ref{sec:BTWE}.

\subsection{Identification of the law of the hull of the origin}

Fix $a \in (0,1)$. Recall from \eqref{eq:charT} the construction of the Uniform Infinite Planar Triangulation (UIPT) as the distributional  local limit of uniform triangulations of  size tending to  infinity.
 Given $T_{\infty}$, we define a \emph{site percolation} (percolation in short) as a random bi-coloring of the vertices of $T_{\infty}$,  obtained by painting independently each vertex white with probability $a$ and black with probability $1-a$, see Fig.~\ref{percoex-site}.  Recall that Angel \cite{Ang03} has proved that the critical threshold parameter for percolation is almost surely $a_{c} := {1}/{2}$ and that, furthermore, at $a_{c}$ there is no percolation.  Angel also proved that on the event that the percolation is nice, the percolation interface going through the root edge is finite in all regimes (subcritical, critical and supercritical) allowing us to perform the necklace surgery. Since almost surely the UIPT has one end, only one of the two hulls $ \mathcal{H}^\circ$ and $ \mathcal{H}^\bullet$ is infinite. In the sequel, we will implicitly use the above remarks without further notice.

To stress the dependence in $ a \in (0,1)$, conditionally on the percolation on the UIPT being nice, we denote by $ \mathcal{H}^\circ_{a}$ and $ \mathcal{H}^\bullet_{a}$ respectively the white and black hulls $ \mathcal{H}^\circ$ and $ \mathcal{H}^\bullet$. We start with a useful remark based on symmetry. 
 \begin{proposition} \label{prop:symmetry} We have the following equality in distribution $ \displaystyle \big( \mathcal{H}^\circ_{a}, \mathcal{H}^\bullet_{a}\big)  \quad\mathop{=}^{(d)} \quad  \big( \mathcal{H}^\bullet_{1-a}, \mathcal{H}^\circ_{1-a}\big)$.\end{proposition}
 \proof This is a consequence of the fact that  flipping all the colors into their opposite reverses the roles of $ \mathcal{H}^\circ$ and $ \mathcal{H}^\bullet$, and exchanges $a$ with $1-a$. \endproof

\begin{proposition} \label {prop:loiH} Let $h \in \mathbb{T}^B$ be a finite triangulation with boundary. Set $ n= \#  \partial h$ . For every $m \geq 1$, we have 
 \begin{eqnarray*} \Pr{ \mathcal{H}^{\circ}_{a} = h,\#\partial\mathcal{H}^\bullet_{a}=m} &=& { 72 \sqrt{6 \pi}} \cdot  12^{n} Q_{a}(h) \cdot {n+m \choose n} \cdot  12^m K_{1-a}(m)   .  \end{eqnarray*}
\end{proposition}
\proof 
For every $N \geq 1$, let $T_{N}$ be a  uniform triangulation with $N$ vertices. Conditionally on $T_{N}$, sample a site percolation on $T_{N}$ with parameter $a \in (0,1)$.
On the event on which the percolation is nice, denote  respectively by $ \mathcal{H}^{\circ}_{a,N}$  and $ \mathcal{H}^{ \bullet}_{a,N}$  the white and black hulls.
Recall the notation $ \mathbb{T}_{n,p}$ for the set of all triangulations with boundary of perimeter $p$ and $n$ vertices in total.
By the necklace decomposition of Section \ref{sec:necklace}, on the event $\{ \# \partial \mathcal{H}^\bullet_{a,N} = m,\mathcal{H}^\circ_{a,N}=h\}$, the triangulation $T_{N}$ is a gluing, along a $(n,m)$-necklace, of the hull $h$ with another triangulation with boundary of perimeter $m$ totalizing $N- \# \mathrm{V}(h)$ vertices,
and such that  all the vertices of the boundary of $h$ are white and those on the boundary of the second triangulation are black.
 Hence,
 \begin{eqnarray*} && \Pr{  \# \partial \mathcal{H}^\bullet_{a,N}=m,\mathcal{H}^{\circ}_{a,N} = h} \\
 && \qquad \qquad \qquad = \frac{1}{ \# \mathbb{T}_{N,1}} a^{\#  \mathcal{E}(h)}  {n+m \choose n} \sum_{t \in \mathbb{T}_{N- \# \mathrm{V}(h),m}}(1-a)^{\# \mathcal{E}(t)}\\
 && \qquad \qquad \qquad = \frac{12^{n}}{ r_{c}^N \# \mathbb{T}_{N,1}} \frac{r_{c}^{\# \mathrm{V}(h)}a^{\#  \mathcal{E}(h)}}{12^{ \# \partial h}} {n+m \choose n} 12^m\sum_{t \in \mathbb{T}_{N- \# \mathrm{V}(h),m}}\frac{r_{c}^{\# \mathrm{V}(t)}(1-a)^{\# \mathcal{E}(t)}}{12^{ \# \partial t}}\\
 && \qquad \qquad  \qquad = \frac{12^{n}}{ r_{c}^N \# \mathbb{T}_{N,1}} Q_{a}(h) {n+m \choose n}12^m Q_{1-a}( \mathbb{T}_{N- \# \mathrm{V}(h),m}).
 \end{eqnarray*}
 Since $T_{N}$ converges locally in distribution towards the UIPT (see \eqref{eq:charT}), using Proposition \ref{prop:phi} and \eqref{eq:equiwnp2} we can take the limit as $N \to \infty$ and get the statement of the proposition.\endproof

 \subsection{The critical exponent  for the perimeter}
 
We are now ready to prove Theorem  \ref {thm:expo}.
 
 \begin {proof}[Proof of Theorem  \ref {thm:expo}]
 
We keep the notation of Section \ref {sec:scoop}, in particular recall that $\mathrm{Tree}(T) $ denotes the tree of components of a triangulation $T \in \mathbb{T}^B$ and that $  \left|\mathrm{Tree}(T) \right|= \# \partial T+1$. Note that when $a =1/2$, we have $r_{c}(2a+ \gamma)/2=1/24$. To simplify notation, set 
 \begin{eqnarray*} \tilde{K}_{n} &:=& 24^{n} \cdot K_{\text{\sfrac{1}{2}}}(n) = \frac{1}{24} \cdot  \mathsf{GW}_{\mu^{\circ}_{\text{\sfrac{1}{2}}},\mu^{\bullet}} \left[  \sum_{u \in \bullet(\tau)}  \phi(k_{u}) \mathbbm{1}_{|\tau|=n+1}\right], \\
 \tilde{Q}_{n} &:=& 24^n \cdot Q_{\text{\sfrac{1}{2}}}( \{ T : \# \partial T = n\}) = \frac{1}{24} \cdot \mathsf{GW}_{\mu^{\circ}_{\text{\sfrac{1}{2}}},\mu^{\bullet}} \left( |\tau|=n+1\right).  \end{eqnarray*}
This implies that
\begin{eqnarray}
\label {eq:asQK}  \qquad \tilde{K}_{n} \quad\mathop{ \sim}_{n \rightarrow \infty} \quad   \frac{ 1}{8 \cdot 3^{5/6} \cdot \Gamma(-2/3)^2 \cdot\sqrt{ 2\pi}} \cdot n^{1/3}\quad \mbox{and } \quad \tilde{Q}_{n}   \quad\mathop{ \sim }_{n \rightarrow \infty} \quad  \frac{1}{8 \cdot 3^{2/3} \cdot |\Gamma(-2/3)|}  \cdot n^{-5/3}.
\end{eqnarray}
Indeed, the first statement follows from \eqref{eq:constante}, while the second one is a consequence of Proposition \ref{prop:gwQ} combined with \eqref{eq:equivtaille}.  Next, using Proposition \ref{prop:symmetry}, write for $n \geq 1$
\begin{eqnarray*}
 &&\hspace{-2cm}\P( \# \partial \mathcal{H}^{\circ}_{\text{\sfrac{1}{2}}} = n) \\
 &=&\P( \# \partial \mathcal{H}^{\circ}_{\text{\sfrac{1}{2}}} = n,|\mathcal{H}^\bullet_{\mbox{\sfrac{1}{2}}}| = \infty)+ \P( \# \partial \mathcal{H}^{\circ}_{\text{\sfrac{1}{2}}} = n, |\mathcal{H}^\bullet_{\text{\sfrac{1}{2}}}| < \infty) \\
  &\underset{ \mathrm{Prop}. \ref{prop:symmetry}}{=}& \P( \# \partial \mathcal{H}^{\circ}_{\text{\sfrac{1}{2}}} = n,|\mathcal{H}^ \bullet_{\text{\sfrac{1}{2}}}| = \infty)
+\P( |\mathcal{H}^{\circ}_{\text{\sfrac{1}{2}}}| < \infty, \# \partial \mathcal{H}^\bullet_{\text{\sfrac{1}{2}}} = n) \\
&=& \sum_{m=1}^\infty \sum_{\begin{subarray}{c}h \in \mathbb{T}^B ,|h|< \infty \\ \# \partial h =n \end{subarray}} \P(  \mathcal{H}^{\circ}_{\text{\sfrac{1}{2}}} = h, \# \partial \mathcal{H}^ \bullet_{\text{\sfrac{1}{2}}} = m) + \sum_{m=0}^\infty \sum_{\begin{subarray}{c}h \in \mathbb{T}^B,|h|< \infty \\ \# \partial h =m \end{subarray}} \P(  \mathcal{H}^{\circ}_{\text{\sfrac{1}{2}}} = h, \# \partial \mathcal{H}^ \bullet_{\text{\sfrac{1}{2}}} = n) \\ & \underset{\mathrm{Prop}.\,\ref{prop:loiH}}{=}& 72 \sqrt{6\pi} {n+m \choose m} 12^{n+m} \cdot \\ && \left( \sum_{m=1}^\infty Q_{\text{\sfrac{1}{2}}}(\{ \mathsf{T} \in \mathbb{T}^B ; \# \partial \mathsf{T} =n\})K_{\text{\sfrac{1}{2}}}(m) + \sum_{m=0}^\infty Q_{\text{\sfrac{1}{2}}}(\{ \mathsf{T} \in \mathbb{T}^B ; \# \partial \mathsf{T} =m\}) K_{\text{\sfrac{1}{2}}}(n) \right)\\
&=& 72 \sqrt{6\pi} \sum_{m=0}^\infty  {n+m \choose n} 2^{-n-m} \big(\tilde{Q}_{m} \tilde{K}_{n} + \tilde{Q}_{n} \tilde{K}_{m}\big),\end{eqnarray*}
with the convention that $ \tilde{K}_{0} =0$.
Now, for every $m \geq 0$ and $u \in \R$, set
$$F_{n}(u)= {2n+ \fl{u \sqrt {n}} \choose n}  \frac{ \sqrt{n}}{2^{2n+ \fl{u \sqrt {n}}}}.$$
It is a simple matter to check that for fixed $ u \in \R$, $F_{n}(u) \rightarrow  {  } e^{- u^{2}/4}/\sqrt { \pi}$ as $n \rightarrow \infty$ and that there exists a constant $C>0$ such that $F_{n}(u) \leq C 2^{-|u|}$ for every $n \geq 2$ and $ u \in \R$. Combined with \eqref{eq:asQK}, the dominated convergence theorem implies that
 \begin{eqnarray*}
 \sum_{ m= 0}^\infty {n+m \choose n}  \frac{1}{2^{m+n}} \frac{\tilde{K}_{m}\tilde{Q}_{n} + \tilde{K}_{n} \tilde{Q}_{m}}{ 4 \tilde{K}_{n}\tilde{Q}_{n}} &=& \int_{- \sqrt{n}}^{ \infty} du \ F_{n}(u) \frac{\tilde{K}_{n+ \fl{u \sqrt{n}}} \tilde{Q}_{n}+ \tilde{K}_{n}\tilde{Q}_{n+ \fl{u \sqrt{n}}}}{ 4 \tilde{K}_{n}\tilde{Q}_{n}}  \\
 & \displaystyle \mathop{\longrightarrow}_{n \rightarrow \infty} &  \frac{1}{2} \cdot \int_{- \infty}^{ \infty} \frac{1}{  \sqrt { \pi}} e^{- u^{2}/4} \ du=  {1}.
 \end{eqnarray*}
 Hence 
 $$  \P( \# \partial \mathcal{H}^{\circ}_{\text{\sfrac{1}{2}}} = n)   \quad\mathop{ \sim}_{n \rightarrow \infty} \quad 288 \sqrt{6 \pi}   \cdot \tilde{K}_{n}\tilde{Q}_{n}.$$
 Thus, by \eqref{eq:asQK}, we get
\begin{eqnarray*}
 \P( \# \partial \mathcal{H}^{\circ}_{\text{\sfrac{1}{2}}} = n)   & \sim  & 32 \cdot 3^{2} \cdot \sqrt{6 \pi} \cdot  \frac{ 1}{8 \cdot 3^{5/6} \cdot \Gamma(-2/3)^2 \cdot\sqrt{ 2\pi}}\cdot n^{1/3} \cdot  \frac{1}{8 \cdot 3^{2/3} \cdot |\Gamma(-2/3)|}   \cdot n^{-5/3}. \\
 &=&  \frac{ {3}}{ {2} \cdot |\Gamma(-{2}/{3})|^3}  \cdot n^{-4/3}.\end{eqnarray*}
 This completes the proof.
 \end {proof}

 \subsection {Scaling limits}
 
 We are now ready to prove Theorem \ref{thm:scalingperco}, which describes the scaling limits of the boundary of large percolation clusters in the UIPT. We start by recalling the definition of the Gromov--Hausdorff topology (see \cite{BBI01} for additional details).

   \subsubsection {The Gromov--Hausdorff Topology}
\label{sec:GH}

If $(E,d)$ and $(E',d')$ are two  compact  metric spaces, the Gromov--Hausdorff distance between
${E}$ and ${E'}$  is defined by
 \begin{eqnarray*} \op{d_{GH}}({E},{E'}) &=& \inf \left\{\op{d}_{\op{H}}^F(\phi(E),\phi'(E'))\right\}, \end{eqnarray*} where the infimum is taken over all choices of metric space $(F,\delta)$ and  isometric embeddings $\phi : E \to F$ and $\phi'  : E' \to F$ of $E$ and $E'$ into $F$, and where $ \mathrm{d}_{ \mathrm{H}}^F$ is the Hausdorff distance between compacts sets in $F$. The  Gromov--Hausdorff distance is indeed a metric on the space of all isometry classes of compact metric spaces, which makes it separable and complete.
 
 An alternative practical definition of  $\op{d_{GH}}$ uses correspondences. A correspondence between two  metric spaces $(E,d)$ and $(E',d')$ is by definition a subset $\mathcal{R} \subset E\times E'$ such that, for every $x_{1} \in E$, there exists at least one point $x_{2}\in E'$ such that $(x_{1},x_{2}) \in \mathcal{R}$ and conversely, for every $y_{2}\in E'$, there exists at least one point $y_{1}\in E$ such that $(y_{1},y_{2}) \in \mathcal{R}$. The distortion of the correspondence $\mathcal{R}$ is defined by 
 $$ \op{dis}(\mathcal{R}) = \sup\big\{|d(x_{1},y_{1})-d'(x_{2},y_{2})| : (x_{1},x_{2}),(y_{1},y_{2}) \in \mathcal{R} \big\}.$$ The Gromov--Hausdorff distance can then be expressed in terms of correspondences by the formula
 \begin{equation}
 \label{GHcorres}
 \op{d_{GH}}({E},{E'})= \frac{1}{2} \inf_{\mathcal{R} \subset E\times E'} \big\{\hspace{-0.5mm}\op{dis}(\mathcal{R})\big\},
 \end{equation} where the infimum is over all correspondences $\mathcal{R}$ between $(E,d)$ and $(E',d')$.

\subsubsection{Discrete looptrees and {scooped-out}  triangulations}
The main ingredient for  proving Theorems \ref{thm:scalingperco} and \ref{thm:nearcrit} is a relation between  the boundary of a triangulation and the discrete looptree associated with its tree of components,  which we now describe. To this end, we need to introduce a slightly modified discrete looptree.  

Let $\tau$ be a plane tree and recall {from the Introduction} the construction of $ \mathsf{Loop}( \tau)$. We define $ \overline{\mathsf{Loop}}(\tau)$ as the graph obtained from $\mathsf{Loop}(\tau)$ by contracting the edges  linking two vertices $u$ and $v$ such that $v$ is the last child of $ u$ in lexicographical order in $\tau$ (meaning that we identify such vertices). 

Recall from Section \ref{sec:scoop} the definition of scooped-out triangulation $\mathsf{Scoop}(T)$ and the tree $\mathsf{Tree}( {T})$ for a triangulation with boundary $ {T}$, and from  Section \ref{section:bijection} the bijection $\mathcal{G}$.

\begin{lemma}  \label{lem:claim} Let  $ T \in \mathbb{T}^B$ be a finite triangulation with boundary. Then the graphs $\overline{ \mathsf{Loop}}\Big(  \mathcal{G}\big(\mathsf{Tree}( T)\big)\Big)$ and $\mathsf{Scoop}( T)$ are equal.
\end{lemma}
\proof[Proof drawings.] A formal proof of this would not be enlightening and this property should be clear on {Fig.~\ref{fig:twig} and Fig.~\ref{fig:twig2}}  below.
\begin{figure}[!h]
  \begin{center}
  \includegraphics[width=0.9 \linewidth]{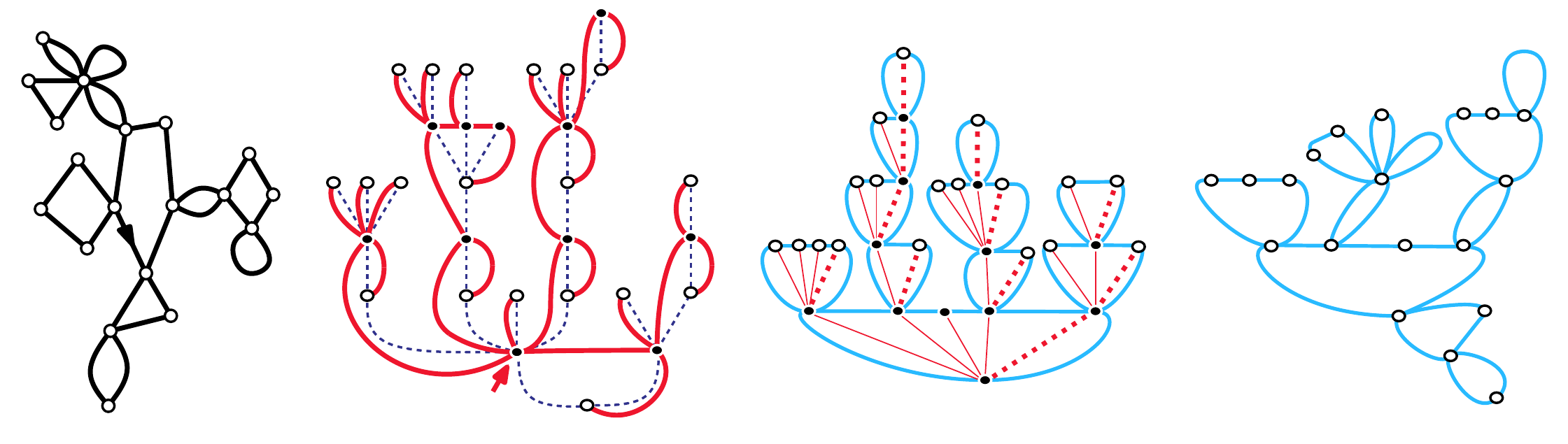}
  \caption{{ The first figure represents a scooped-out triangulation $\mathsf{Scoop}(T)$, the second one represents its associated trees $ \mathsf{Tree}(T)$ (with dashed edges) and $ \mathcal{G}( \mathsf{Tree}( T))$ (in bold red), the third one represents ${\mathsf{Loop}}(\mathcal{G}( \mathsf{Tree}( T)))$ (in light blue). Finally, the last figure represents  $ \overline{\mathsf{Loop}}(\mathcal{G}( \mathsf{Tree}( T)))$ (which is exactly $\mathsf{Scoop}(T)$).}  \label{fig:twig}}
 \end{center}
  \end{figure}
     \begin{figure}[!h]
  \begin{center}
  \includegraphics[width=0.9 \linewidth]{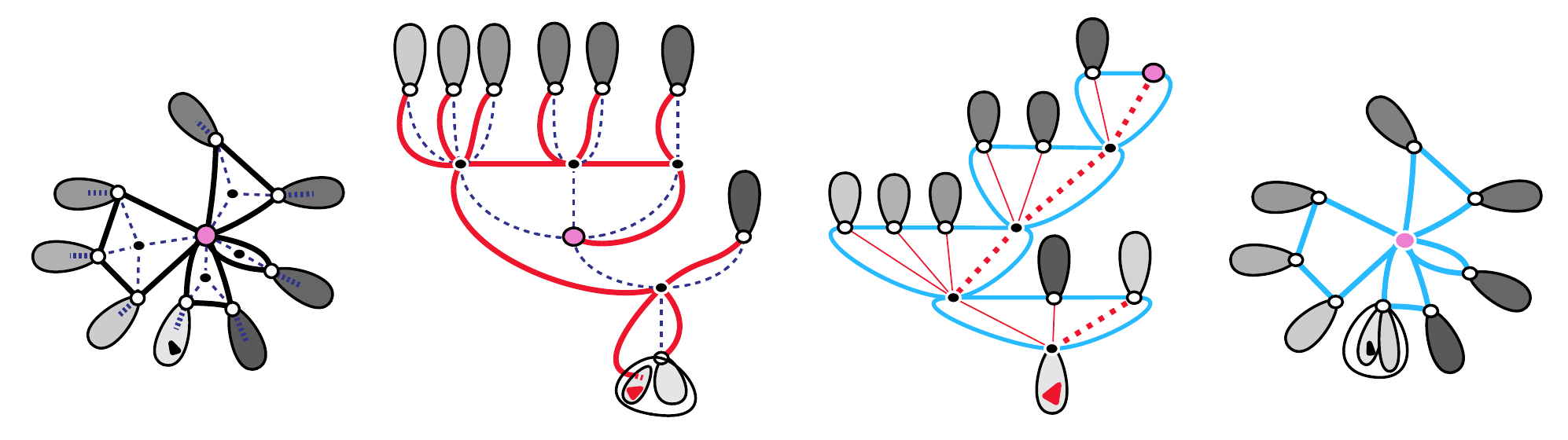}
  \caption{The same elements as in Fig.~\ref {fig:twig}, but locally around a pinch-point of $T$.
  \label{fig:twig2}}
  \end{center}
  \end{figure} \endproof
  
  \subsubsection {Proof  of Theorems \ref{thm:scalingperco} and \ref{thm:nearcrit}}

Fix $a \in (0,1)$ and an integer $n \geq 1$. Let $\mathcal{H}_a^{\circ}(n)$ denote the random variable $\mathcal{H}^\circ_{a}$  conditioned on the event $  \{ \# \partial \mathcal{H}^\circ_{a}=n, |\mathcal{H}^\circ_{a}| < \infty \}$.
From Proposition \ref{prop:loiH}, it follows that the distribution of $\mathcal{H}^\circ_{a}(n)$  is the probability measure $Q_{a}( \,  \cdot \,  | \,   \# \partial T =n)$. Hence, by Proposition \ref{prop:gwQ} the tree of components $\mathsf{Tree}(\mathcal{H}_a^{\circ}(n))$ is distributed as  a $\mathsf{GW}_{\mu^{\circ}_{a},\mu^{\bullet}}$ tree conditioned on having $n+1$ vertices. Set $ \tau^n_a= \mathcal{G}(\mathsf{Tree}(\mathcal{H}_a^{ \circ} (n)))$. By Proposition \ref{prop:onetype}, $ \tau^n_a$ is distributed as a $ \mathsf{GW}_{ \nu_{a}}$ tree conditioned on having $n+1$ vertices.

\medskip

\begin{proof}[Proof  of Theorem \ref{thm:scalingperco}]
\emph {Case $1/2<a<1$.} Denote by $ \Delta_{n}$ the maximal degree of $ \tau^{n}_{a}$ and let $u^{ \star}_{n} \in \tau^{n}_{a}$  be a vertex with degree $ \Delta_{n}$ (this vertex is asymptotically unique by \cite[Theorem 5.5]{JS11}). Hence, by \cite[Theorem 5.5]{JS11},
$$ \frac{ \Delta_{n}}{n}  \quad\mathop{\longrightarrow}^{( \P)}_{n \rightarrow \infty} \quad 1- \sum_{i=0}^ \infty i \nu_{a}(i)=  \frac{2a-1}{ \sqrt{3}-1+2a}.$$
Since $ \mu^{ \circ}_{a}$ has an exponential tail, a simple argument shows that $u^{ \star}_{n}$ is a black vertex with probability tending to one as $n \rightarrow \infty$. Informally, this means that a loop of length roughly $ \Delta_{n}$ appears in  $ \partial \mathcal{H}_a^{ \circ} (n)$. By \cite[Corollary 2]{K12+},  the maximal size of the connected components of $ \tau^n_a \backslash  \{ u^{ \star}_{n}\}$, divided by $n$, converges in probability towards $0$ as $n \rightarrow \infty$. By properties of the bijection $ \mathcal{G}$, this means that the maximal size of a cluster branching on the macroscopic loop corresponding to the black vertex $ u^{ \star}_{n}$, {divided by $n$, converges in probability towards $0$ as $n \rightarrow \infty$}, and immediately implies that
$$ \frac{1}{n} \cdot \partial \mathcal{H}_a^{ \circ} (n)  \quad\mathop{\longrightarrow}^{(d)}_{n \rightarrow \infty} \quad  \frac{2a-1}{ \sqrt{3}-1+2a} \cdot \mathcal{C}_{1}.$$

\emph {Case $a=1/2$.} First of all recall that $ \partial \mathcal{H}^\circ _{a}(n)$, viewed as a metric space, is the same as  $\mathsf{Scoop}( \mathcal{H}^\circ_{a}(n))$, viewed as a metric space. We shall thus {work with}  the latter.  Recall from \eqref{eq:nu} that in the  case $a = 1/2$,   $ \nu= \nu_{\text{\sfrac{1}{2}}}$ is critical and  $ \nu(k)   \sim \frac{ \sqrt{3}}{4 \sqrt{\pi}} k^{-5/2}$ as $k \rightarrow \infty$.  Next, since the longest path in $ \tau^n_{\text{\sfrac{1}{2}}}$ containing only vertices {that are identified in the definition of $\overline{\mathsf{Loop}}(\tau^n_{\text{\sfrac{1}{2}}})$}  is bounded by the height $\mathsf{H}( \tau^n_{\text{\sfrac{1}{2}}})$ {of the tree $\tau^n_{\text{\sfrac{1}{2}}}$}, this implies that
$$\mathrm{d_{GH}} \left(\overline{\mathsf{Loop}}(\tau^n_{\text{\sfrac{1}{2}}}), \mathsf{Loop}(\tau^n_{\text{\sfrac{1}{2}}}) \right)   \leq 2 \mathsf{H}(\tau^n_{\text{\sfrac{1}{2}}}).$$ Since $\tau^n_{\text{\sfrac{1}{2}}}/n^{1/3}$ converges towards the stable tree of index $3/2$ (see \cite{Du03} {or \cite{K11}}), we get that the quantity $\mathsf{H}( \tau^n_{\text{\sfrac{1}{2}}})/ n^{2/3}$ converges in probability towards $0$ as $n \rightarrow \infty$. These observations combined with \eqref{eq:invprinc} show that $$ n^{-2/3} \cdot \overline{\mathsf{Loop}}( \tau^n_{\text{\sfrac{1}{2}}}) \quad \xrightarrow[n\to\infty]{(d)}\quad  3^{1/3} \cdot \mathscr{L}_{ 3/2},$$  {where the convergence holds} in distribution for the Gromov--Hausdorff topology. Finally, we use Lemma \ref{lem:claim} to replace $\overline{\mathsf{Loop}}( \tau^n_{\text{\sfrac{1}{2}}})$ by $ \mathrm{Scoop}( \mathcal{H}^\circ_{\text{\sfrac{1}{2}}}(n))$ in the last display and get the desired result.

\emph {Case $a<1/2$.} In this case, the tree $\tau^{n}_{a}$ is a supercritical Galton--Watson tree conditioned on having $n+1$ vertices. {We perform a standard exponential tilting of the offspring distribution in order to reduce to the critical case as follows}. Recall that $F_{a}$ denotes the generating function of the offspring distribution $\nu_{a}$. For every $\lambda \in (0,1)$ it is easy to see  that $\tau^n_{a}$ has the same law as a Galton--Watson tree whose offspring distribution generating function is $ z \mapsto F_{a}(\lambda z)/F_{a}(\lambda)$ {conditioned on having $n+1$ vertices}, see e.g. \cite{Ken75}.
We now choose $ \lambda_{a} \in (0,1)$ be the unique positive real number such that \begin{equation} \label{def:lambdaa} \lambda_{a} \cdot F_{a}'( \lambda_{a})=F_{a}( \lambda_{a}).  \end{equation} In other words, { $\lambda_{a}$ is chosen such that} the offspring distribution $ \widetilde{ \nu}_{a}$ whose generating function is $z \mapsto F_{a}( \lambda_{a}z)/F_{a}(\lambda_{a})$ is critical {(see e.g. \cite[Section 4]{Jan12} for a proof that $ \lambda_{a}$ exists and is unique)}. We write $\widetilde{\tau}^n_{a}$ for a $\widetilde{\nu}_{a}$-Galton--Watson tree conditioned on having $n+1$ vertices. Since $ \widetilde{ \nu}_{a}$ has small exponential moments, we can apply \cite[Theorem 14]{CHK13} and get 
  \begin{eqnarray*} \frac{1}{ \sqrt{n}}\cdot \overline{\mathsf{Loop}}( \tau^{n}_{a}) \overset{(d)}{=} \frac{1}{ \sqrt{n}}\cdot \overline{\mathsf{Loop}}( \widetilde{\tau}^{n}_{a})  & \xrightarrow[n\to\infty]{(d)} &    { \frac{2}{ \widetilde{\sigma}_{a}}} \cdot   \frac{1}{4}\left( \widetilde{\sigma}_{a}^2+ \widetilde{ \nu}_{a}(2 \Z_{+}) \right) \cdot  \mathcal{T}_{ \mathbf{e}}, \end{eqnarray*}
where $  \widetilde{\sigma}_{a}^2$  is the variance of $\widetilde{ \nu}_{a}$ and $\widetilde{ \nu}_{a}(2 \Z_{+}) =\widetilde{ \nu}_{a}(0)+\widetilde{ \nu}_{a}(2)+\widetilde{ \nu}_{a}(4)+ \cdots$. Applying Lemma \ref{lem:claim}, we have established Theorem \ref {thm:scalingperco} in the case $a<1/2$ with $C_{a}= { \frac{2}{ \widetilde{\sigma}}} \cdot   \frac{1}{4}\left( \widetilde{\sigma}_{a}+ \widetilde{ \nu}_{a}(2 \Z_{+}) \right)$.  \end{proof}

\begin{proof}[Proof of Theorem \ref{thm:nearcrit} ]
It follows from the proof of Theorem \ref{thm:scalingperco} that:
 \begin{equation}
 \label{eq:Ca}C_{a}=\displaystyle \frac{2a-1}{ \sqrt{3}-1+2a} \quad a \in (1/2,1),  \qquad C_{a}= { \frac{2}{ \widetilde{\sigma}_{a} }} \cdot   \frac{1}{4}\left( \widetilde{\sigma}_{a}^2+ \widetilde{ \nu}_{a}(2 \Z_{+}) \right) \quad a \in (0,1/2).
 \end{equation} The asymptotic estimate as $a \downarrow 1/2$ immediately follows. As $a \uparrow 1/2$, we have $\widetilde{\sigma}_{a} \rightarrow \infty$, so that $ C_{a} \sim \widetilde{\sigma}_{a}/2$. Next,  using the exact expression of $F_{a}$ given in \eqref{eq:nua}, simple calculations show that  $ \lambda_{a}=c_{a}^{1/3}+c_{a}^{-1/3}-1$ with
  $$c_{a}=8 (1-a) a-4 (1-2a) \sqrt{a-a^2} i-1.$$
  Note that $c_a$ is a complex number, but that $\lambda_a$ is real. In particular we have  $$\lambda_{a}  =1- \frac{16}{9} \cdot (1/2-a)^2 +o((1/2-a)^2), \qquad {F_{a}( \lambda_{a})}  \mathop{\longrightarrow}_{a \uparrow 1/2}   1, \qquad  F_{a}''( \lambda_{a})  \mathop{ \sim}_{a \uparrow 1/2}  \frac{3 \sqrt {3}}{16} \cdot \frac{1}{1/2-a}.$$
For the last asymptotic estimate, we use the fact that
$$F''_{a}(z)= \frac{3}{4(2a-1+ \sqrt{3})} \cdot \frac{1}{ \sqrt{1-z}}$$
Since $\widetilde{\sigma}^2_a =  {\lambda_{a}^2 \cdot F_{a}''( \lambda_{a})}/{F_{a}( \lambda_{a})}$, the conclusion follows.
\end{proof}

 \section{Comments}
 
 \subsection{Critical Boltzmann triangulations} 
 Instead of working with the infinite model of the UIPT we can also consider another natural model of random triangulations: the critical Boltzmann measure is the probability measure on triangulations which assigns a probability$$  \P_{b}(T)=\frac{1}{ \sum_{ t \in  \mathbb{T}} r_{c}^{\# \mathrm{V}(t)}} \cdot r_{c}^{\# \mathrm{V}(T)}$$
to every finite triangulation $T$, where $r_{c}=1/ \sqrt {432}$ and $ \mathbb{T}$ is the set of all finite triangulations. Note that $\sum_{ t \in  \mathbb{T}} r_{c}^{\# \mathrm{V}(t)}=(192 \sqrt{3})^{-1}$ by \eqref{eq:Wtildeexact}. With this model, the underlying triangulation $ \mathsf{T}$ is always finite, so that both the white and black hulls $ \mathcal{H}^{\circ}_{a}$ and $ \mathcal{H}^{\bullet}_{a}$ are finite. It is straightforward to adapt Proposition \ref{prop:loiH} in order to get the following:

\begin{proposition} \label {prop:loiH2} Let $h$ be a finite triangulation with boundary of perimeter $n$. We have 
 \begin{eqnarray*} \PrB{ \mathcal{H}^{\circ}_{a}=h} &=& 192 \sqrt{3} \cdot 12^n Q_{a}(h) \sum_{m=0}^\infty 12^m  {n+m \choose n} {Q_{1-a}( \{ T : \# \partial T = m\})}.
 \end{eqnarray*}
 \end{proposition}
 As for the UIPT, this implies that the random variable $\mathcal{H}^{\circ}_{a}$, under $ \P_{b}$, conditioned on the event $  \{ \# \partial \mathcal{H}^\circ_{a}=n\}$, is distributed according to $Q_{a}( \,  \cdot \,  | \,   \# \partial T =n)$. Hence Theorems \ref{thm:scalingperco} and \ref{thm:nearcrit} remain true without changes. 
To get an analog of Theorem \ref{thm:expo},  we similarly compute  \begin{eqnarray*} \PrB{ \#\partial \mathcal{H}^{\circ}_{a_{c}} = n} &=& 192 \sqrt{3} \cdot \tilde{Q}_{n}  \cdot   \sum_{m \geq 0}  {n+m \choose n} \frac{1}{2^{m+n}} \tilde{Q}_{m}\\
&  \underset{n \rightarrow \infty}{ \sim} &   384 \sqrt{3} \cdot \tilde{Q}_{n} ^2= \frac{2 \cdot 3^{1/6}}{\Gamma(-2/3)^2} \cdot n^{-10/3}. \end{eqnarray*}

\begin{rek}\label{rek:boltz} It is useful to note that the exponent $10/3$ appearing in the last formula is obtained as  \begin{eqnarray} \label{eq:bolts} 2 \left( 1 + \frac{1}{\alpha}\right)\end{eqnarray} with $\alpha = 3/2$. {Note also that} $1 + {1}/\alpha$ is exactly the exponent appearing in the probability that a Galton--Watson tree with an offspring distribution in the domain of attraction of an $\alpha$-stable law  has a large progeny (see Proposition \ref {prop:general} (i) for a precise statement). 
\end{rek}

\subsection{Type II triangulations}
\label{sec:universality}
 In this work, we focused on general triangulations, but similar results can be derived in the context of type II triangulations, that are triangulations where no loops are allowed. The approach is exactly the same, the necklace surgery and the tree representation of clusters work alike and so we only give the main intermediate enumerative results. In particular, if  $ \overline{W}$ denotes the generating function of type II triangulations with simple boundary with weight $x$ per \emph{inner} vertex and $y$ then $ \overline{W}(x,y)$ is also well-known (see \cite{GJ83}). In particular, the radius of convergence of $\overline{W}$ as a function of $x$ is $\overline{r}_{c} = {2}/{27}$, and we have
  \begin{eqnarray*} {\overline{W}}(\overline{r}_{c},y) &=& \frac{y}{2}+\frac{(1-9 y)^{3/2}-1}{27}. \end{eqnarray*}
As previously,  for every integer $k \geq 1$, set $\overline{q}_{k} = [y^k] { \overline{W}}(\overline{r}_{c},y) / 9^k$.  Note that $\overline{q}_{1}=0$. In this case, the tree of components of the white hull is described similarly as in Proposition \ref{prop:gwQ} by the two offspring distributions $\overline{\mu}^{\circ}_{a}$ and $ \overline{\mu}^{\bullet}$ defined by  $$\overline{\mu}^{\bullet}(j) = \frac{\overline{q}_{j+1}}{\overline{Z}_{\bullet}}, \qquad   \overline{\mu}^{\circ}_{a}(j) = (1-\xi)\xi^j  \qquad (j \geq 0),$$
 where  $\overline{Z}_{ \bullet}=1/54$ and $ \xi=1/(1+4a)$. Note that $\overline{\mu}^{\bullet}(0)=0$, which is consistent with the fact that we are working with type II triangulations, since black vertices with no children of the tree of components are in bijection with loops.  In addition, if $ \overline{\nu}_{a}$ is the  image of $ \mathsf{GW}_{\overline{\mu}^{\circ}_{a},\overline{\mu}^{\bullet}}$ by $ \mathcal{G}$, then
$$ \sum_{i \geq 0} \overline{\nu}_{a}(i) z^i=  \frac{4 a-2+3z+2 (1-z)^{3/2}}{4 a+1}.$$
In particular the mean of $\overline{\nu}_{a}$ is $3/(1+4a)$, so that $ \overline{\nu}_{a}$ is critical if and only if $a=1/2$. When $a=1/2$, to simplify notation we write $ \overline{\nu}= \overline{\nu}_{\text{\sfrac{1}{2}}}$. Note that then: 
$$\sum_{i \geq 0} \overline{\nu}(i) z^i=z+\frac{2}{3}(1-z)^{3/2}, \qquad \qquad \nu(k)    \quad\mathop{ \sim }_{k \rightarrow \infty}  \quad  \frac{ 1}{2 \sqrt{\pi}} \cdot k^{-5/2}.$$
 It easily follows that Theorem \ref {thm:scalingperco} holds in this case with the constants
 $C_{\text{\sfrac{1}{2}}}= (3/2)^{2/3}$, $C_{a}=4(a-1/2)/(4a+1)$ for $1/2<a<1$. In addition, 
 $$C_{a}  \quad  \underset{ a \downarrow 1/2}{\sim} \quad \frac{4}{3} \left(a-\frac{1}{2} \right) \qquad  \qquad \textrm{and} \qquad \qquad C_{a}\underset{a \uparrow 1/2}{ \sim} \quad \displaystyle   \frac{ \sqrt{3}}{4 \sqrt {2}} \cdot  \left( \frac{1}{2}-a \right)^{-1/2}.$$
Theorem \ref{thm:expo} also holds in this case with the same exponent (but with a different constant).

  \subsection{Conjectures about $O(N)$ models on random triangulations}
  \label{sec:conjectures}
 In this work, we {established} that $ \mathscr{L}_{3/2}$ {is the} scaling limit of the boundary of the cluster of the origin for critical site percolation on random triangulations (such as the UIPT or random Boltzmann triangulations). We conjecture that all the {family of} looptrees $ \mathscr{L}_{\alpha}$ for $\alpha \in (1,2)$ {are} scaling limits of boundary clusters of {certain} statistical mechanics models on  random planar maps. 

More precisely, we focus on the so-called $O(N)$ model  on random planar triangulations. We follow closely the presentation of \cite[Section 8]{LGM11} and \cite[Section 3.4]{Mie12}, see also \cite{BBG11,BBG12}. A loop configuration $ \ell$ on a triangulation $ \mathsf{T}$ is a collection of loops drawn on the dual of $ \mathsf{T}$ such that two different loops visit different faces. Equivalently, {a loop configuration is a consistent}  gluing of two types of triangles (an empty triangle, and a triangle with a {dual} path inside {joining}  two different edges) such that the result is a topological sphere, see Fig.~\ref{fig:boucles}.

\begin{figure}[!h]
 \begin{center}
 \includegraphics[width=8cm]{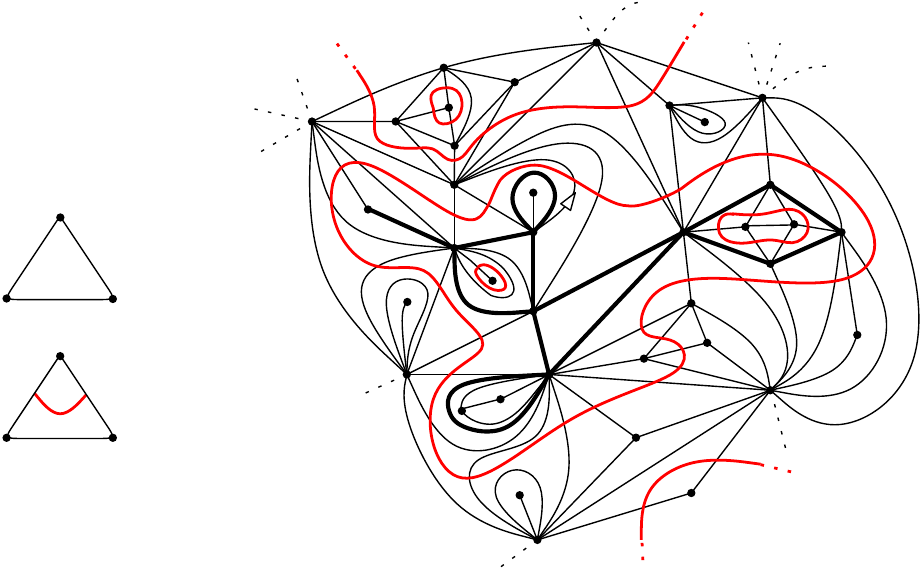}
 \caption{A triangulation with a loop decoration and the cluster of the origin \label{fig:boucles}.}
 \end{center}
 \end{figure}
 
The total perimeter $|\ell|$ of the collection of loops is the number of faces visited by the union of the loops and the number of loops of $\ell$ is denoted by $\# \ell$.  
Provided that a loop traverses the root edge ({which we assume from now on}), we can define the hull of the origin and its boundary as depicted in Fig.~\ref{fig:boucles}. 
We can also define the gasket of $\{ \mathsf{T}, \ell\}$ {as}  the map obtained by removing the interior of the loops (the exterior of a loop contains the target of the root edge) as well as the faces traversed by the loops, see \cite[Fig.~4]{BBG11}.  

 The Boltzmann annealed $O(N)$ measure on decorated triangulations is the probability measure
$$ P_{g,h,N}( \{ \mathsf{T}, \ell\}) = \frac{1}{Z_{g,h,N}} g^{ \# \mathsf{F}( \mathsf{T})} h^{|\ell|} N^{\# \ell} ,$$
where $g,h,N >0$ {are parameters}  and $Z_{g,h,N}$ is a normalizing constant. Now fix $N \in (0,2)$. For most {choices} of {the parameters} $g,h$, large random triangulations under $P_{g,h,N}$ are believed to converge towards the Brownian map unless $g$ and $h$ {satisfy a certain relation} {which we assume to hold }  from now on, see \cite{BBG11,BBG12} for similar models.
We are thus left with one parameter $h >0$. In this case, Le Gall and Miermont \cite{LGM11} provided the conjectural scaling limits of the \emph{gasket} of decorated triangulations: It is conjectured that there exists $h_{c}(N)$ so that if $h < h_{c}(N)$ then the scaling limits of random planar (decorated) triangulations as well as their gaskets converge towards the Brownian map.
If $h > h_{c}(N)$, {it is conjectured that}  the gasket of a large random decorated triangulation under $P_{g,h,N}$ converges (after suitable scaling) towards the stable map of parameter $$a = 3/2 + \pi^{-1} \mathrm{arcsin}(N/2).$$
In this regime, stable maps  have large macroscopic faces that touch themselves and each other.
In the discrete underlying model, this means that the boundaries of large clusters should possess pinch-points at large scale.
We believe that in this case, the scaling limits of these cluster boundaries (or equivalently the inner geometry of a face in a stable map of parameter $a \in (3/2,2)$) is {the stable} looptree $ \mathscr{L}_{\alpha}$ {of index $ \alpha$, where $ \alpha$ satisfies the relation}    \begin{eqnarray} \label{relationaalpha} 1 + \frac{1}{\alpha} &=& a \quad \in (3/2,2).  \end{eqnarray}

{Let us give an heuristic argument supporting this prediction.} In {\cite[Eq.~3.18]{BBG11}} the authors showed that (in certain closely related models) the probability that the origin hull $ \mathcal{H}$ of a random decorated map {under  under $P_{g,h,N}$ (with the parameters chosen as above)} satisfies
$$  \mathbb{P} ( \# \partial \mathcal{H} = k)  \quad\mathop{ \sim}_{k \rightarrow \infty} \quad  C \cdot k^{-2a}.$$
On the other hand, if the tree structure associated to the origin hull is a {critical}  Galton--Watson tree with an offspring distribution in the domain of attraction of an $\alpha$-stable law, from Remark \ref{rek:boltz} we should have 
$$  \mathbb{P} ( \# \partial \mathcal{H} = k) \quad\mathop{ \sim}_{k \rightarrow \infty} \quad C \cdot k^{-2(1+\alpha^{-1})}.$$
Identifying the exponents in the last two displays gives  {our conjecture \eqref{relationaalpha}} .

 \appendix
 \section{Appendix: proof of the technical lemmas}
 
 We conclude this work by establishing several technical results, some of which involve stable densities. We will only use the case $\alpha={3}/{2}$, but prove the general case in view of future applications. By \emph{$ \alpha$-stable Lévy process} we will always mean a 
stable spectrally
positive Lévy process $(X_{t})_{t \geq 0}$ of index $\alpha$, normalized so that for every
$\lambda>0$, $\E[\exp(-\lambda X_t)]=\exp(t \lambda^\alpha)$.
The process $X$ takes values in the Skorokhod space $\D(\R_+, \R)$ of
right-continuous with left limits (càdlàg) real-valued functions,
endowed with the Skorokhod topology (see \cite[Chap. 3]{Bil99}). The dependence of $X$ in $\alpha$ will be implicit in this section, and we denote by $p_{1}$ the density of $X_{1}$.
 
 
 \subsection{Technical lemmas on stable densities}

The following result, which  is a consequence of \cite[Lemma XVII.6.1]{Fel71}, will be useful.

\begin {lemma} \label {lem:densite} We have $p_{1}(0)=1/| \Gamma(-1/ \alpha)|$.
\end{lemma}

\begin {lemma}\label {lem:technique2}For every $ \beta>0$ and $ \alpha \in (1,2)$,
$$ \int_{0}^{ \infty} dx \ x^{ \beta} \cdot  p_{1}(-x) dx= \frac{ \Gamma( \beta)}{ \Gamma( \beta/ \alpha)}.$$
\end {lemma}

\proof 
By \cite[Lemma XVII.6.1]{Fel71},  we have the following series representation, valid for $x>0$:
$$p_{1}(-x)=  -\frac{1}{ \pi x} \sum_{k=1}^{ \infty} \frac{ \Gamma(1+k/\alpha)}{k!} (-x)^{k} \sin( k \pi/\alpha).$$
To simplify notation, set $F(A)= \int_{0}^{ A} dx \ x^{ \beta} \cdot  p_{1}(-x) dx$ for $A>0$, so that:
\begin{eqnarray*}
F(A)&=& - \frac{1}{ \pi} \sum_{k=1}^{ \infty} \frac{ \Gamma(1+k/\alpha)}{k!} \frac{1}{k+\beta} A^{k+\beta} (-1)^{k} \sin(k \pi/\alpha) \\
&=& -\frac{A^ \beta}{ \pi} \cdot \mathsf{Im} \left(  \sum_{k=1}^{ \infty}  \frac{ \Gamma( k/ \alpha+1) \Gamma(k+ \beta)}{ \Gamma(k+ \beta+1)}  \cdot \frac{ \left(-A e^{i \pi/ \alpha} \right)^k}{k!} \right).
\end{eqnarray*}
By \cite[(2.3.10)]{PK11}, we have
$$\mathsf{Im} \left(  \sum_{k=1}^{ \infty}  \frac{ \Gamma( k/ \alpha+1) \Gamma(k+ \beta)}{ \Gamma(k+ \beta+1)}  \cdot \frac{ \left(-A e^{i \pi/ \alpha} \right)^k}{k!} \right)  \quad\mathop{ \sim}_{A \rightarrow \infty} \quad - \Gamma( \beta) \Gamma(1- \beta/ \alpha) \sin( \pi \beta / \alpha) \cdot A^{- \beta}.$$
Hence, using the reflection formula for the $ \Gamma$ function, we get
$$ \lim_{A \rightarrow \infty} F(A)= \Gamma( \beta) \Gamma(1- \beta/ \alpha) \frac{\sin( \pi \beta / \alpha) }{ \pi} = \frac{\Gamma( \beta) }{ \Gamma( \beta/ \alpha)}. \qedhere $$

In particular, note that for $ \alpha=3/2$ and $ \beta=1/2$ we have\begin{equation}
p_{1}(0)=\frac{2}{3 \Gamma(1/3)} \quad \mbox{ and }\quad  \label{eq:int}\int_{0}^{ \infty} dx \sqrt {x} \cdot  p_{1}(-x) dx= \frac{ \sqrt { \pi}}{ \Gamma(1/3)}.
\end{equation}

\subsection {A technical estimate for Galton--Watson trees}

We next establish  the following asymptotic estimates.

 \begin {proposition} \label {prop:general} Fix $ \alpha \in (1,2)$ and $ \beta> \alpha$. Let $ \rho$ be a critical offspring distribution such that $ \rho_{k} \sim C \cdot k^{-1-\alpha}$ as $ k \rightarrow \infty$ for a certain $C>0$. Let $\phi : \mathbb{Z}_{+} \to \mathbb{R}_{+}$ a function such that $ \phi(x) \sim \kappa \cdot x^{\beta}$ as $x \rightarrow \infty$ for a certain $ \kappa>0$.
 
 \begin{enumerate}
  \item[(i)] We have $$ {\mathsf{GW}}_{ \rho} \left(| \tau|=n \right) \quad\mathop{ \sim}_{n \rightarrow \infty} \quad  \frac{1}{| \Gamma(-1/ \alpha)|  \cdot ( \Gamma(- \alpha) C)^{1/ \alpha}} \cdot \frac{1}{n^{1+1/ \alpha}}.$$
  \item[(ii)] We have
 $$  \mathsf{GW}_{ \rho}\left[ \left. \sum_{u \in \tau} \phi(k_{u})  \right | |\tau|=n  \right]  \quad\mathop{ \sim}_{n \rightarrow \infty} \quad \kappa \cdot C^{( \beta-1)/ \alpha} \cdot \Gamma(- \alpha)^{( \beta- \alpha -1)/ \alpha}\cdot  \frac{ \Gamma( \beta- \alpha -1)}{ \Gamma(( \beta- \alpha -1)/ \alpha)} \cdot  n^{\beta/\alpha}.$$
 \item[(iii)] We have
 \begin{eqnarray*}
&& \mathsf{GW}_{ \rho}\left[  \sum_{u \in \tau} \phi(k_{u})  \mathbbm{1}_{ |\tau|=n}  \right]  \\
&& \quad\mathop{ \sim}_{n \rightarrow \infty} \quad \kappa \cdot C^{( \beta-2)/ \alpha} \cdot \Gamma(- \alpha)^{( \beta- \alpha -2)/ \alpha}\cdot  \frac{ \Gamma( \beta- \alpha -1)}{ | \Gamma(-1/\alpha)| \Gamma(( \beta- \alpha -1)/ \alpha)} \cdot  n^{(\beta- \alpha-1)/\alpha}.
 \end{eqnarray*}
   \end{enumerate}
\end {proposition}

Before proving Proposition \ref {prop:general}, we state two useful results.

\begin{theorem}[Local limit theorem]\label{thm:locallimit}
Let $(W_n)_{n \geq 0}$ be a random walk on $\Z$ started from $0$. Assume that $ \P[W_{1}<0]=0$ and that there exists $ \alpha \in (1,2)$ and $c>0$ such that $ \P[W_{1}=k] \sim C \cdot k^{ -1-\alpha}$ as $k \rightarrow \infty$. Set $a_{n}= \left(\Gamma(- \alpha) C \right)^{1/ \alpha}n^{1/ \alpha}$. Then:
\begin{equation}
\label{eq:ll}\lim_{n \rightarrow \infty} \sup_{k \in \Z}
\left| a_n \P[W_n=k]-p_1\left( \frac{k-n \E[W_1]}{a_n}\right)
\right|=0.
\end{equation}
\end{theorem}

See e.g. \cite[Theorem 4.2.1]{IL71} for a proof of the local limit theorem.

\begin {lemma} \label {lem:GW}Let $(S_{n})_{n \geq 0}$ be a random walk on $ \Z$ whose jump distribution is $  \rho$. Then:
\begin{enumerate}
 \item[(i)] for $n \geq 1$, $ \displaystyle \GWrho{| \tau|=n} = \frac{1}{n} \Pr{S_{n}=n-1} $
 \item[(ii)] for every function $F:  \Z \rightarrow \R_{+}$,
$$ \Esrho { \left. \sum_{u \in \tau} F(k_{u})  \right| {| \tau|=n} }= n \cdot   \Es {F(S_{1}) | S_{n}=n-1}.$$
  \end{enumerate}
\end {lemma}

\begin {proof}
For (i), see e.g. \cite[Section 5.2]{Pit06}. Assertion (ii) easily follows from the fact that $\sum_{u \in \tau} F(k_{u})$ is invariant under cyclic shifts.
\end {proof}

We are now ready to prove Proposition  \ref{prop:general}.
\begin {proof}[Proof of Proposition \ref{prop:general}]
Let $(S_{n})_{n \geq 0}$ be a random walk on $ \Z$ whose jump distribution is $  \rho$. 
Since $ \rho$ is critical and $ \rho_{i} \sim   C i^{-1- \alpha}$  as $i \rightarrow \infty$, \eqref{eq:ll} applies with $a_{n}= \left(\Gamma(- \alpha) C \right)^{1/ \alpha}n^{1/ \alpha}$. Then by Lemma \ref {lem:GW} (i) and the local limit theorem:
\begin{equation*}
\GWrho{| \tau|=n} \quad\mathop{ \sim}_{n \rightarrow \infty} \quad  \frac{ p_{1}(0)}{ ( \Gamma(- \alpha) C)^{1/ \alpha}}  \cdot \frac{1}{n^{1+1/ \alpha}}  = \frac{1}{| \Gamma(-1/ \alpha)|  \cdot ( \Gamma(- \alpha) C)^{1/ \alpha}} \cdot \frac{1}{n^{1+1/ \alpha}}.
\end{equation*}
This proves the first assertion.

For (ii), write
\begin{eqnarray}
 &&  \frac{1}{n^{ \beta/ \alpha}}\Esrho { \left. \sum_{u \in \tau} \phi(k_{u})  \right| {| \tau|=n} }\notag \\
  &&  \qquad \qquad = n \sum_{k=0}^{ \infty}  \frac{\phi(k)}{n^ { \beta/ \alpha}}  \rho_{k} \frac{ \Pr {S_{n-1}=n-1-k}}{ \Pr {S_{n}=n-1}}  \qquad \textrm{by Lemma \ref {lem:GW} (ii)}\notag\\
  && \qquad \qquad= \int_{0}^ \infty dx \frac{\phi(\fl {x n^{1/ \alpha}})}{n^ {\beta/ \alpha}}   n^{1+1/ \alpha}\rho_{ \fl {x n^{1/ \alpha}}} \frac{ \Pr {S_{n-1}=n-1-\fl {x n^{1/ \alpha}}}}{ \Pr {S_{n}=n-1}}. \label {eq:cvdom}
\end{eqnarray}
In order to use the dominated convergence theorem we use the following technical lemma whose proof is postponed to the end of this section:
\begin {lemma} \label {lem:bornesousexp} Let $ \xi$ be a critical probability measure on $  \{0,1,2,  \ldots\}$ of span $1$ (i.e. the greatest integer dividing all the integers $n$ such that $ \xi(n)>0$ is $1$). Let $F(z)= \Es {z^{ \xi}}$ be the probability generating function of $ \xi$ and assume there exists $ c \in \R$ such that the following Taylor expansion holds around $z=1$:
\begin{equation}
\label{eq:hyp}F(z)=1+(z-1)+ c (z-1)^{ \alpha} +o(|z-1|^{ \alpha}) ,\qquad |z| \leq 1.
\end{equation}
Let $S_{N}= \sum_{i=1}^{N} \xi_{i}$, where $ \xi_{i}$ are i.i.d. copies of $ \xi$. There exists constants $c_{1},c_{2}$ such that for every $k, N \geq 1$:
$$ \P(S_{N}=N-k) \leq \frac{c_{1}}{N^{1/ \alpha}} e^{-c_{2} k^{ \alpha}/N}.$$
\end {lemma}
\noindent We return to the proof of Proposition \ref{prop:general}. We may apply Lemma \ref {lem:bornesousexp}, which combined with the local limit theorem gives that for every $ n \geq 1$ and $x \geq 0$
$$ \frac{ \Pr {S_{n-1}=n-1-\fl {x n^{1/ \alpha}}}}{ \Pr {S_{n}=n-1}} \leq c_{3}  e^{-c_{4} x^{ \alpha}}.$$
In addition, since $ \phi(x) \sim  \kappa \cdot x^{\beta}$ and $ \rho_{k} \sim C \cdot k^{-1- \alpha}$, we have
$$\frac{\phi(\fl {x n^{1/ \alpha}})}{n^{ \beta/ \alpha}}   n^{1+1/ \alpha}\rho_{ \fl {x n^{1/ \alpha}}} \leq c_{5} x^{ \beta- \alpha-1}.$$
The expression appearing under the integral in \eqref{eq:cvdom} is thus bounded by $c_{6} x^{ \beta- \alpha-1} e^ {-c_{4} x^{  \alpha}}$. By using the dominated convergence theorem as $n \rightarrow \infty$ in \eqref{eq:cvdom} combined with the local limit theorem,  we conclude that:
\begin{eqnarray*}
\frac{1}{n^{ \beta/ \alpha}}\Esrho { \left. \sum_{u \in \tau} \phi(k_{u})  \right| {| \tau|=n} }  & \displaystyle \mathop{\longrightarrow}_{n \rightarrow \infty} & \kappa C \int_{0} ^{ \infty} dx  \  {x}^{ \beta- \alpha -1} p_{1} \left( -  \frac{x}{(\Gamma(-\alpha)C)^{1/\alpha}} \right) \\
&=&  \kappa C  \left( \Gamma(- \alpha)C \right)^{( \beta- \alpha-1)/ \alpha} \cdot \int_{0}^ \infty dx\  {x}^{ \beta- \alpha -1} p_{1}(-x). \\
&=& \kappa C  \left( \Gamma(- \alpha)C \right)^{( \beta- \alpha-1)/ \alpha} \cdot  \frac{ \Gamma( \beta- \alpha -1)}{ \Gamma(( \beta- \alpha -1)/ \alpha)}\\
&=&  \kappa \cdot C^{( \beta-1)/ \alpha} \cdot \Gamma(- \alpha)^{( \beta- \alpha -1)/ \alpha}\cdot  \frac{ \Gamma( \beta- \alpha -1)}{ \Gamma(( \beta- \alpha -1)/ \alpha)}.
\end{eqnarray*}
This completes the proof of (ii).

The last assertion should now be of no difficulty. 
\end {proof}

Recall the two offspring distributions $ \mu^{\bullet}$ and $\mu^{\circ}_a$ introduced in Section \ref{section:twotype} and the bijection $ \mathcal{G}$ introduced in Section \ref{section:bijection}. Here we take $a=1/2$, and to simplify notation we set $\mu^{ \circ}=\mu^{ \circ}_{\text{\sfrac{1}{2}}}$. Recall finally that $ \nu$ is the image of $ \mathsf{GW}_{\mu^{\circ}_{a},\mu^{\bullet}}$ by $ \mathcal{G}$. In the previous sections, we have used the following corollary of Proposition \ref{prop:general}:

\begin{corollary} \label {cor:At}Let  $ \phi$ be the same function as in Proposition \ref {prop:phi}. We have the following two asymptotic behaviors:
\begin{eqnarray*}
\mathsf{GW}_{\mu^{ \circ},\mu^{\bullet}} \left(| \tau|=n \right)=\GWnu{| \tau|=n} & \displaystyle \mathop{ \sim}_{n \rightarrow \infty} &  \frac{3^{1/3}}{|\Gamma(-2/3)|} \cdot n^{-5/3} \\
 \GWE { \sum_{u \in  \bullet(\tau)} \phi(k_{u})  \mathbbm{1}_{| \tau|=n} }  & \displaystyle \mathop{ \sim}_{n \rightarrow \infty} & \frac{3^{1/6}}{  \Gamma(-2/3)^2 \cdot\sqrt{ 2\pi}} \cdot n^{1/3}.\end{eqnarray*}
\end {corollary}

\begin{proof} By Proposition \ref{prop:onetype}, using the fact that $ \mathcal{G}$ maps black vertices of degree $k$ to vertices of degree $k+1$, it is sufficient to prove that
$$  \Esnu { \sum_{u \in \tau} \phi(k_{u}-1)  \mathbbm{1}_{| \tau|=n} }  \quad\mathop{ \sim}_{n \rightarrow \infty} \quad \frac{3^{1/6}}{  \Gamma(-2/3)^2 \cdot\sqrt{ 2\pi}} \cdot n^{1/3},$$
where by convention we set $ \phi(-1)=0$. Corollary \ref {cor:At} hence immediately follows from Proposition \ref{prop:general} (iii), applied with $C= \sqrt{3/ \pi}/4$, $ \kappa= {4}/{9} \cdot   \sqrt{{6}/{ \pi}}$, $ \alpha=3/2$ and $ \beta=3$.\end{proof}

\begin {proof}[Proof of Lemma \ref{lem:bornesousexp}]
We adapt the proof of Lemma 2.1 of \cite{Jan06b}. By the local limit theorem we may assume that $ N^{1/ \alpha}\leq k \leq N$. To simplify notation, set $G(z)=F(z)/z$. By \eqref{eq:hyp}, we have the following Taylor expansion around $z=1$: $G(z)=1+c (z-1)^ \alpha+o(|z-1|^ \alpha)
$ for $|z| \leq 1$.
In particular, 
\begin{equation}
\label{eq:taylor} \ln G(e^w)=c w^ \alpha+o(|w|^ \alpha), \qquad \mathsf{Re } \ w \leq 0.
\end{equation}
By integration around the circle of radius $ \exp(- \delta k^{ \alpha-1}/N)$ (for some small $ \delta>0$ to be chosen later):
\begin{eqnarray}
\P(S_{N}=N-k) &=& \frac{1}{2 i\pi} \oint z^{k-N} F(z)^N \frac{dz}{z} \notag\\
&=& \frac{1}{2 \pi} \int_{ - \pi}^ \pi\exp(- \delta k^{ \alpha}/N +ik t) G(e^{- \delta  k^{ \alpha-1}/N+it})^N dt \label{eq:expr}\end{eqnarray}
It is easy to check that, for $ \eta, t \geq 0$ small enough, $ \mathsf {Re}( (- \eta+it)^ \alpha) \leq  \eta^ \alpha- |\cos( \alpha \pi/2)| t^ \alpha$ and $|- \eta+it|^{ \alpha} \leq 2( \eta^ \alpha+t^ \alpha)$.
Thus, by \eqref{eq:taylor},
\begin{eqnarray*}
\ln|G(e^{- \delta  k^{ \alpha-1}/N+it})| &=& \mathsf{Re } \ln G(e^{- \delta  k^{ \alpha-1}/N+it})\\
&=&c \left( - \delta  k^{ \alpha-1}/N+it \right)^ \alpha+o \left(\left| - \delta  k^{ \alpha-1}/N+it \right|^ \alpha \right) \\
& \leq & c \delta^ \alpha  k^{ \alpha ( \alpha-1)}/N^{ \alpha}-c|\cos( \alpha \pi/2)| t^ \alpha+ o \left(  k^{ \alpha ( \alpha-1)}/N^{ \alpha}+ t^ \alpha\right)
\end{eqnarray*}
Since $k^{ \alpha( \alpha-1)}/ N^ { \alpha} \leq N^{ \alpha ( \alpha-1)}/ N^{ \alpha}= N^{- \alpha(2- \alpha)} \rightarrow 0$ as $N \rightarrow \infty$, there exist $ \delta_{0}, t_{0}>0$ sufficiently small such that for $ 0< \delta \leq  \delta_{0}$ and $|t| \leq t_{0}$:
\begin{equation}
\label{eq:t1}\ln|G(e^{- \delta  k^{ \alpha-1}/N+it})| \leq 2c \delta^ \alpha  k^{ \alpha ( \alpha-1)}/N^{ \alpha}-c|\cos( \alpha \pi/2)| t^ \alpha/2.
\end{equation}

Now, since the span of $ \xi$ is $1$, we have  $|F(z)| <1$ for $|z| \leq 1$ and $z \neq 1$, so by continuity and compactness, there exists $ \epsilon \in (0,1)$ such that $|F(re^{it})| \leq 1- \epsilon < e^{- \epsilon}$ when $e^{- \delta_{0}} \leq r \leq 1$ and $t_{0} \leq |t| \leq  \pi$. Thus, using the fact that $k^{ \alpha-1} \leq  N$, for $t_{0} \leq |t| \leq  \pi$ and $0 \leq  \delta \leq \delta_{1} := \min( \delta_{0}, \epsilon/2)$ we get:
\begin{equation}
\label{eq:t2}|G(e^{- \delta k^{ \alpha-1}/N+it})| =e^{\delta k^{ \alpha-1}/N}|F(e^{- \delta k^{ \alpha-1}/N+it})| \leq e^{ \delta} e^{- \epsilon} \leq  e^{- \epsilon/2}.
\end{equation}
Combining \eqref{eq:t1} and \eqref{eq:t2}, we get for $ \delta \leq  \delta_{1}$ and $|t| \leq  \pi$:
$$|G(e^{- \delta  k^{ \alpha-1}/N+it})| \leq e^{2c \delta^ \alpha  k^{ \alpha ( \alpha-1)}/N^{ \alpha}- c_{1} t^ \alpha}$$
with $c_{1}:= \min(c|\cos( \alpha \pi/2)|/2, \epsilon/(2 \pi^\alpha))$. Plugging this in
\eqref{eq:expr}, we get:
$$\P(S_{N}=N-k)  \leq e^{2c \delta^ \alpha  k^{ \alpha ( \alpha-1)}/N^{ \alpha-1}- \delta k^{ \alpha}/N } \int_{- \infty}^{ \infty} e^{-c_{1} Nt^ \alpha}dt.$$
The result follows by choosing $ \delta \leq 1/(2c)^{ \alpha-1}$. Indeed, for $\delta \leq 1/(2c)^{ \alpha-1}$, since $k \geq N^{1/ \alpha}$, we have
$$ \left( \frac{k^ \alpha}{N} \right)^{2- \alpha} \geq 2c \delta^{ \alpha-1}.$$
This completes the proof. \end {proof}
  

 \end{document}